\documentclass[12pt,oneside]{amsart}
\usepackage{amssymb,amscd,amsthm,verbatim}
\usepackage[margin=1.2in]{geometry}
\usepackage{latexsym}
\usepackage{color}
\usepackage{xcolor}
\usepackage{graphicx}
\usepackage{mathrsfs}
\usepackage[all,cmtip]{xy}
\usepackage{bbm}
\usepackage{stmaryrd}
\usepackage{xr}

\usepackage[
	colorlinks,
	linkcolor={black},
	citecolor={black},
	urlcolor={black}
]{hyperref}

\date{\today}

\newcommand{\C}{{\mathbb C}}

\newcommand{\M}{{\mathfrak M}}

\newcommand{\N}{{\mathbb{N}}}
\newcommand{\bbN}{{\mathbb{N}}}

\newcommand{\PP}{{\mathbb P}}

\newcommand{\Q}{{\mathbb Q}}

\newcommand{\R}{{\mathbb R}}

\newcommand{\T}{{\mathbb T}}

\newcommand{\X}{{\mathbb S}}

\newcommand{\Z}{{\mathbb Z}}

\newcommand{\oT}{{\overline{T}}}

\newcommand{\og}{{\overline{g}}}

\newcommand{\CW}{{\mathcal{W}}}

\newcommand{\A}{{\mathcal A}}

\newcommand{\SL}{{\mathrm{SL}}}

\newcommand{\dist}{{\mathrm{dist}}}

\newcommand{\hull}{{\mathrm{hull}}}

\newcommand{\orb}{{\mathrm{orb}}}

\newcommand{\Ran}{{\mathrm{Ran}}}

\newcommand{\rot}{{\mathrm{rot}}}

\newcommand{\sgn}{{\mathrm{sgn}}}

\newcommand{\tr}{{\mathrm{Tr}\,}}

\newcommand{\unif}{{\mathrm{unif}}}

\renewcommand{\ker}{{\mathrm{Ker}}}

\newcommand{\schwartzmanGroup}{{\mathfrak{A}}}

\newtheorem{theorem}{Theorem}[section]
\newtheorem{lemma}[theorem]{Lemma}
\newtheorem{prop}[theorem]{Proposition}
\newtheorem{coro}[theorem]{Corollary}

\theoremstyle{definition}
\newtheorem{claim}{Claim}[theorem]
\newtheorem{step}{Step}[theorem]
\newtheorem{case}{Case}[theorem]

\newtheorem{definition}[theorem]{Definition}
\newtheorem{example}[theorem]{Example}
\newtheorem{remark}[theorem]{Remark}

\newtheorem{question}[theorem]{Question}
\theoremstyle{plain}

\newenvironment{claimproof}[1][Proof of Claim]{\noindent \underline{#1.} }{\hfill$\diamondsuit$}

\allowdisplaybreaks 

\numberwithin{equation}{section}

\DeclareMathOperator{\supp}{supp}

\newcommand{\set}[1]{\left\{#1\right\}}
\newcommand{\chisub}[1]{\chi_{_{#1}}}

\definecolor{purple}{rgb}{.5,0,1}
\newcommand{\jdf}[1]{\textcolor{red}{#1}}

\begin{document}

\title[Gap Labelling for 1D Schr\"odinger Operators]{Gap Labelling for Discrete One-Dimensional \\ Ergodic Schr\"odinger Operators}

\author{David Damanik}
\thanks{D.D.\ was supported in part by NSF grants DMS--1700131 and DMS--2054752 and Simons Fellowship $\# 669836$}

\author{Jake Fillman}
\thanks{J.F.\ was supported in part by Simons Foundation Collaboration Grant \#711663.}

\dedicatory{Dedicated to the memory of Sergey Naboko}

\maketitle

\begin{abstract}
In this survey, we give an introduction to and proof of the gap labelling theorem for discrete one-dimensional ergodic Schr\"odinger operators via the Schwartzman homomorphism. To keep the paper relatively self-contained, we include background on the integrated density of states, the oscillation theorem for 1D operators, and the construction of the Schwartzman homomorphism. We illustrate the result with some examples. In particular, we show how to use the Schwartzman formalism to recover the classical gap-labelling theorem for almost-periodic potentials. We also consider operators generated by subshifts and operators generated by affine homeomorphisms of finite-dimensional tori. In the latter case, one can use the gap-labelling theorem to show that the spectrum associated with potentials generated by suitable transformations (such as Arnold's cat map) is an interval.
\end{abstract}

\setcounter{tocdepth}{1}
\tableofcontents

%changing the hyperlink-colors such that they are are black in the table of contents and thereafter dark blue
\hypersetup{
	linkcolor={black!30!blue},
	citecolor={black!30!blue},
	urlcolor={black!30!blue}
}

\section{Introduction}
\subsection{Setting}

We will discuss Schr\"odinger operators of the form $H = H_V = \Delta +V$ acting in $\ell^2(\Z)$, where
\[ [\Delta\psi](n) = \psi(n-1)+\psi(n+1), \quad [V\psi](n) = V(n)\psi(n). \]
The main case of interest is that in which the potential, $V$, is dynamically defined over a \emph{topological dynamical system}. Given a compact metric space $\Omega$, a homeomorphism $T:\Omega \to \Omega$, and a continuous function $f: \Omega \to\R$, we consider 
\begin{equation} \label{eq:HfomegaDef} H_{f,\omega} = \Delta + V_{f,\omega},
\end{equation} where the potential $V_{f,\omega}$ is defined by
\begin{equation} \label{eq:VfomegaDef}
V_{f,\omega}(n) = f(T^n\omega).
\end{equation}
The Schr\"odinger operators that can be defined in this manner have been heavily studied over the years; see for example the books \cite{CarmonaLacroix1990, Chulaevsky1989:APO, CFKS, ESO1, PasturFigotin1992:ESO, Teschl2000:Jacobi}, the surveys \cite{Damanik2007KotaniSurvey, Damanik2017ESOSurvey, Jitomirskaya2007, Kotani1997, MarxJito2017ETDS, Spencer1990}, and references therein. One reason for the popularity of this perspective is that this setup contains as special cases numerous models of particular interest including, inter alia, periodic operators ($\Omega = \Z_p$, $T\omega = \omega+1 \ \mathrm{mod} \ p$), quasi-periodic operators ($\Omega = \T^d$, $T$ is translation by a vector with rationally independent coordinates), almost-periodic operators ($\Omega$ is a compact monothetic group and $T$ is a minimal translation),  random operators ($\Omega$ is a suitable product space and $T$ is a shift thereupon), and subshift operators ($\Omega$ is a subshift of low complexity and $T$ is the shift on $\Omega$). We will discuss each of these examples in more detail later in the paper.

It is often fruitful to equip the topological dynamical system $(\Omega,T)$ with additional structure. A Borel probability measure $\mu$ is said to be $T$\emph{-invariant} if $\mu(B) = \mu(T^{-1}B)$ for every Borel $B \subseteq \Omega$ and $T$\emph{-ergodic} if it is $T$-invariant and $\mu(B)\in\{0,1\}$ whenever $T^{-1}B = B$. We will also simply say that $(\Omega,T,\mu)$ is ergodic. It is well known and not hard to show that for a given $(\Omega,T)$, the space of invariant Borel probability measures is nonempty, convex, and compact (in the weak$^*$ topology), that the ergodic measures are the extreme points of the set of invariant probability measures, and that there is always at least one ergodic measure \cite{Walters1982:ErgTh}.

Ergodic systems enjoy many helpful statistical properties; for the purpose of this text, the chief consequences of ergodicity that we employ are  the almost-sure constancy of invariant functions and the Birkhoff pointwise ergodic theorem. That is, for any Borel function $g$, if $g \circ T = g$, then there is a constant $c$ for which $g(\omega) = c$ for $\mu$-a.e.\ $\omega \in \Omega$, and, if $h \in L^1(\Omega,\mu)$, then
\begin{equation}
\frac{1}{N} \sum_{n=0}^{N-1} h(T^n\omega) \to \int_\Omega \! h \, d\mu \quad \text{ for } \mu\text{-a.e.\ } \omega \in \Omega.
\end{equation}

Given $(\Omega,T)$ and $f \in C(\Omega,\R)$, and defining $\{H_{f,\omega}\}$ as above, if $\mu$ is a $T$-ergodic Borel probability measure on $\Omega$, then one sees that the spectrum is $\mu$-almost everywhere constant, that is, there exists a compact set $\Sigma_{\mu,f} \subseteq \R$ such that 
\begin{equation}
\sigma(H_{f,\omega})= \Sigma_{\mu,f} \text{ for }\mu\text{-a.e.\ }\omega \in \Omega.
\end{equation}
 In a similar way, the absolutely continuous, singular continuous, and pure point spectra of $H_{f,\omega}$ are $\mu$-almost everywhere constant \cite{CFKS, KunSou1980CMP}. If $(\Omega,T)$ is uniquely ergodic, then we simply write $\Sigma_f = \Sigma_{\mu,f}$ where $\mu$ is the unique invariant measure. Similarly, if $(\Omega,T)$ is minimal, then there exists $\Sigma_f$ such that $\sigma(H_{f,\omega}) = \Sigma_f$ for \emph{all} $\omega \in \Omega$ and again there is no need to note the dependence on the measure.

The \emph{density of states measure} $\kappa = \kappa_{\mu,f}$ associated with the family $\{H_\omega\}_{\omega \in \Omega}$ is given by
\[ \int g \, d\kappa = \int_\Omega \langle \delta_0, g(H_\omega) \delta_0 \rangle \, d\mu(\omega). \]
It is well known that 
\begin{equation} \label{eq:SigmaDOSrelation}
\Sigma_{\mu,f} = \supp(\kappa_{\mu,f}).
\end{equation}
The accumulation function of $\kappa$,
\[ k(E) = k_{\mu,f}(E) = \int \! \chi_{_{(-\infty,E]}} \, d\kappa, \]
is called the \emph{integrated density of states} (IDS).  We will discuss proofs of some fundamental facts about the IDS in Section~\ref{sec:IDS}. Our presentation is based on Benderskii--Pastur \cite{BenderskiiPastur1970}, Pastur \cite{Pastur1980CMP}, Avron--Simon \cite{AvrSim1983DMJ}, and Delyon--Souillard \cite{DelSou1984CMP}. We direct the reader to the surveys \cite{Kirsch1985LNM, KirschMetzger2007PSPM, Simon1987DOS} for more background about the density of states and related objects.

The present survey aims to discuss the \textbf{gap labelling theorem} for one-dimensional operators: Given an ergodic topological dynamical system $(\Omega,T,\mu)$, there is a countable group $\schwartzmanGroup = \schwartzmanGroup(\Omega,T,\mu) \subseteq \R$ such that for any continuous $f \in C(\Omega,\R)$, 
\begin{equation}k_{\mu,f}(E) \in \schwartzmanGroup \cap [0,1] \end{equation} 
for all $E \in \R \setminus \Sigma_{\mu,f}$, and moreover, this group can be explicitly calculated in many examples of interest (e.g.\ for almost-periodic and random operator families). The precise formulation follows. The relevant definitions take some work to set up. For definitions of the terms used in the statement of the theorem (suspension, homotopy classes, and the Schwartzman homomorphism), we refer the reader to Definitions~\ref{def:flow:suspension}, \ref{def:homotopy}, \ref{def:flow:schwarzmannhom}.

\begin{theorem}[Gap-Labelling Theorem] \label{t:gablabel}
Given an ergodic topological dynamical system $(\Omega,T,\mu)$ such that $\supp\mu = \Omega$, let $(X,\oT,\overline{\mu})$ denote its suspension, let $C^\sharp(X,\T)$ denote the set of homotopy classes of maps $X \to \T$, let $\mathfrak{A}_{\overline{\mu}}: C^\sharp(X,\T) \to \R$ denote the Schwartzman homomorphism, and denote by
\begin{equation}
\schwartzmanGroup = \schwartzmanGroup(\Omega,T,\mu) := \mathfrak{A}_{\overline{\mu}}(C^\sharp(X,\T))
\end{equation}
the range of the Schwartzman homomorphism. Then, for any continuous $f \in C(\Omega,\R)$, 
\begin{equation}k_{\mu,f}(E) \in \schwartzmanGroup \cap [0,1] \end{equation} 
for all $E \in \R \setminus \Sigma_{\mu,f}$.
\end{theorem}

One can show (cf.\ Remark~\ref{rem:gl:AcontainsZ}) that the group $\schwartzmanGroup(\Omega,T,\mu)$ always contains $\Z$; furthermore, the reader can check from the definitions that the IDS obeys $0 \le k \le 1$. Consequently, the gap labelling theorem has the following important consequence. If one can show that $\schwartzmanGroup = \Z$, then the only labels corresponding to open gaps are $0$ and $1$, which correspond to the trivial gaps $(-\infty,\min\Sigma_{\mu,f})$ and $(\max\Sigma_{\mu,f},\infty)$. Thus, if $\schwartzmanGroup(\Omega,T,\mu) = \Z$, then $\Sigma_{\mu,f}$ is an interval for \emph{every} continuous $f:\Omega \to \R$.

\begin{coro} \label{coro:A=Z}
 Suppose $(\Omega,T,\mu)$ is an ergodic topological dynamical system. If $\schwartzmanGroup(\Omega,T,\mu)=\Z$, then $\Sigma_{\mu,f}$ is an interval for every $f \in C(\Omega,\R)$. More generally, if $\schwartzmanGroup(\Omega,T,\mu) = n^{-1}\Z$ for some $n \in \N$, then $\Sigma_{\mu,f}$ is a finite union of no more than $n$ intervals for every $f \in C(\Omega,\R)$. \end{coro}

\begin{remark} Let us make some additional comments about Theorem~\ref{t:gablabel}.
\begin{enumerate}
\item The Schwartzman homomorphism derives its name from the study of the asymptotic cycle by Schwartzman \cite{Schwarzmann1957Annals}.
\medskip

\item The theorem as formulated here and the approach to the proof we give are due to Johnson \cite{Johnson1986JDE}. There is a more general gap-labelling framework obtained by Bellissard and coauthors that uses K-theory of C$^*$-algebras and was in fact obtained earlier than Johnson's work \cite{Bel1986, Bel1990, Bel1992b, Bel2003, BelBovGhe1992}. See Section~\ref{subsec:ktheory} for further discussion.
\medskip

\item Each connected component of $\R\setminus\Sigma_{\mu,f}$ is called a \emph{gap} of the spectrum. In view of \eqref{eq:SigmaDOSrelation}, the IDS is constant on every gap of the spectrum. As such, Theorem~\ref{t:gablabel} implies that there is a function that associates an element of $\schwartzmanGroup\cap[0,1]$ to each gap of $\Sigma_{\mu,f}$. Accordingly, elements of $\schwartzmanGroup(\Omega,T,\mu)$ are called \textit{gap labels}.
\medskip

\item In the event in which $(\Omega,T)$ is uniquely ergodic, we suppress $\mu$ and simply write $\schwartzmanGroup(\Omega,T)$ instead of $\schwartzmanGroup(\Omega,T,\mu)$.
\medskip

\item One can show that $C^\sharp(X,\T)$ is countable (see Proposition~\ref{prop:flow:homotopyClassFacts}). Consequently, the set $\schwartzmanGroup(\Omega,T,\mu)$ is countable. In particular, the set of gap labels is always countable, and therefore could possibly be in one-to-one correspondence with the set of gaps.
\medskip
 
\item The theorem is optimal in the sense that there exist models for which every possible gap label corresponds to an open gap. That is, there exist topological dynamical systems $(\Omega,T,\mu)$ and $f \in C(\Omega,\R)$ such that for every $a \in \schwartzmanGroup(\Omega,T,\mu) \cap [0,1]$, there is a nonempty open interval $I \subseteq \R$ such that  $k|_I \equiv a$. In this situation, one says \textit{all gaps are open}.
\medskip

\item In some situations, one can try to show that all gaps are open for generic choices of $f \in C(\Omega,\R)$ \cite{ABD12, Eliasson1992, PS06, Simon1976AIHP}. See Section~\ref{subsec:deriveCantor} for further discussion.
\medskip

\item One may also try to prove that all gaps are open for specific dynamical systems and explicit sampling functions. The most famous version of this problem concerns the almost-Mathieu operator, where $\Omega = \T$, $T\omega = \omega + \alpha$ with $\alpha$ irrational, and $f(\omega) = 2\lambda \cos(2\pi \omega) $. In this case, it is known that the set of possible labels is $\schwartzmanGroup = \{n+m\alpha : n,m \in \Z\}$ (see Theorem~\ref{t:ap:qpschwartzmangroup} for details). It is widely expected that for any $\lambda \neq 0$ and any irrational $\alpha$, all gaps are open. This is known as the \emph{Dry Ten Martini Problem}, and has been discussed in \cite{AviJit2010JEMS, BorgniaSlager, CEY90, LiuYuan2015, Puig2004, Riedel2012ATMP} and in the forthcoming work \cite{AvilaYouZhouPreprint}. A similar result (all possible gaps open) has been presented for the Extended Harper's Model in the non-self-dual region for Diophantine $\alpha$ as well \cite{Han2018TAMS}. 
\medskip

\item This problem (showing that all possible gap labels correspond to open gaps) has also been heavily studied for substitution Hamiltonians \cite{BaaGriJos1993, Luck1989PRB}, such as those generated by the Fibonacci \cite{DamGor2011CMP, DamGorYes2016Invent, Raymond1997}, period-doubling \cite{BelBovGhe1991, LiuQuYao2021Preprint}, and Thue--Morse \cite{Bel1990} Hamiltonians. 
\end{enumerate}
\end{remark}

Beyond Schr\"odinger operators, gap-labelling is also relevant in other contexts, such as the study of the asymptotic distribution of the Lee--Yang zeros in the Ising model; compare \cite{BaaGriPis1995JSP, DFLY2015IMRN, GrimmBaake1995proc}.

\subsection{Examples}

Let us see some examples. We will give proofs in later sections.

The first class of examples we discuss are those generated by minimal translations of compact groups. Recall that a toplogical dynamical system $(\Omega,T)$ is called \emph{minimal} if the $T$-orbit of every point is dense in $\Omega$. Minimal translations of groups are important since they are precisely the dynamical systems by which one recovers almost-periodic sequences.

\begin{example}[Group Translations] \label{ex:ap} If $G$ is an abelian group and $g \in G$, we write $R_g:G \to G$ for the map $R_g: x \mapsto g+x$. If $G$ is a compact abelian group with a dense cyclic subgroup, we say that $G$ is \emph{monothetic}. In the event that $G$ is monothetic and $g$ generates a dense subgroup, it is known that $(G,R_g)$ is uniquely ergodic, with normalized Haar measure supplying the unique invariant measure. These examples will be addressed in detail with proofs in Section~\ref{sec:ap}.
\begin{enumerate}
\item For $p \in \N$, let $\Z_p := \Z / p\Z$ denote the integers modulo $p$; for $k \in \Z$, we write $[k] \in \Z_p$ for the residue class of $k$. We have
\begin{equation}
\schwartzmanGroup(\Z_p,R_{[1]}) = p^{-1}\Z.
\end{equation}
This gives the first precise instance of a general heuristic: that $\schwartzmanGroup(\Omega,T)$ represents the ``frequencies'' of $(\Omega,T)$ in a suitable sense.
\item Given $d \in \N$, let $\T^d = \R^d/\Z^d$ denote the $d$-dimensional torus. We say $\alpha \in \R^d$ has rationally independent coordinates if the set $\{1,\alpha_1,\alpha_2,\ldots,\alpha_d\}$ is linearly independent over $\Q$. Equivalently, for $k \in \Z^d$,
\[ \sum_{j=1}^d k_j \alpha_j  \in \Z \implies k_1=k_2= \cdots = k_d = 0. \]
In this case, $(\T^d,R_\alpha)$ is strictly ergodic\footnote{We employ a minor but standard abuse of notation, writing $\alpha$ both for the element of $\R^d$ and its projection to $\T^d$. This will be convenient later on when we talk about flows.} and (defining $\alpha_0=1$ for convenience)
\begin{equation}\label{e.RalphaSchwartzman}
\schwartzmanGroup(\T^d,R_\alpha) = \Z+\alpha\Z^d = \set{\sum_{j=0}^d k_j \alpha_j : k_j \in \Z \ \forall 0 \le j \le d}.
\end{equation}
\item In general, suppose $\Omega$ is a compact monothetic group and $\alpha \in \Omega$ generates a dense cyclic subgroup $\langle \alpha \rangle = \{n\alpha:n \in \Z\}$. Assuming $\Omega$ is not cyclic, $\langle \alpha \rangle \cong \Z$.  Letting $\widehat{\Omega}$ denote the dual group (see Section~\ref{sec:ap} for definitions), we have an injective homomorphism $\widehat{\varphi}_\alpha: \widehat{\Omega} \to \T$ given by $\widehat{\varphi}_\alpha(\chi) = \chi(\alpha)$. Denoting $\pi:\R \to \T$ the canonical projection, one has
\begin{equation}
\schwartzmanGroup(\Omega,R_\alpha) = \pi^{-1}(\widehat{\varphi}_\alpha(\widehat{\Omega})).
\end{equation}
This allows one to recover the classical gap-labelling theorem for almost-periodic potentials; compare \cite{DelSou1983CMP, JohnMos1982CMP}.
\end{enumerate}
\end{example}

\begin{example}[Subshifts] \label{ex.subshift}
The next class of examples come from \emph{subshifts}, which give a natural setting in which one may study ergodic potentials taking finitely many values. Given a finite set $\A$, called the \emph{alphabet}, the \emph{full shift} is the space $\A^\Z$ equipped with the product topology. The \emph{shift} acts on $\A^\Z$ by $[T\omega](n)=\omega(n+1)$. A \emph{subshift} is a nonempty closed (hence compact) and $T$-invariant subset $\Omega \subseteq \A^\Z$. These examples will be addressed in detail with proofs in Section~\ref{sec:subshift}.
\begin{enumerate}
\item Suppose $(\Omega,T)$ is a subshift with ergodic measure $\mu$. Then
\begin{equation}\label{e.subshiftSchwartzman}
\schwartzmanGroup(\Omega,T,\mu) = \set{\int f \, d\mu : f \in C(\Omega,\Z)}.
\end{equation}
Thus, for a given subshift, the problem of determining the set of gap labels amounts to identifying integrals of elements of $C(\Omega,\Z)$ against $\mu$ which in turn can be reduced to computing the $\mu$-measure of cylinder sets (compare Theorem~\ref{t.subshiftproperties}).
\item Let $\A = \{0,1\}$, define $S:\A \to \A^*$ by $S:0 \mapsto 01$ and $1\mapsto 0$, let $u_n = S^n(0)$, and define
\[ \Omega = \set{\omega \in \A^\Z : \text{ for every interval }I \subseteq \Z, \ \omega|_I \text{ is a subword of some } u_n }.  \]
This is a strictly ergodic subshift, called the \emph{Fibonacci subshift}. For additional background about subshifts generated by substitions, see Section~\ref{subsec:substitution} and \cite{Queffelec1987}. One has
\begin{equation}
\schwartzmanGroup(\Omega,T) = \Z + \alpha\Z = \{m+n\alpha: m,n\in\Z\},
\end{equation}
where $\alpha = \frac{1}{2} (\sqrt{5}-1)$ is the inverse of the golden mean.

\item Consider $\Omega = \{1,2,\ldots,m\}^\Z$, $T:\Omega \to \Omega$ the shift $[T\omega](n) = \omega(n+1)$, $\mu_0$ a measure on $\{1,2,\ldots,m\}$, and $\mu = \mu_0^\Z$ the product measure. Then $\schwartzmanGroup(\Omega,T,\mu) $ is the $\Z$-algebra generated by $\{\mu(\{j\}) : 1 \le j \le m\}$, that is,
\begin{align}
\schwartzmanGroup(\Omega,T,\mu) 
= \Z[w_1,\ldots,w_n]  
= \set{\sum_{k \in \Z_+^n} c_k w^k : c_k \in \Z \text{ and } c_k = 0 \text{ for all but finitely many } k},
\end{align}
where $w_j = \mu_0(\{j\})$ and $w^k = w_1^{k_1}w_2^{k_2} \cdots w_n^{k_n}$ for $k \in \Z_+^n$.
\end{enumerate}
\end{example}

\begin{example}[Affine Toral Homeomorphisms] \label{ex:affine}
Given $d \in \N$, consider $\Omega = \T^d = \R^d/\Z^d$, the $d$-dimensional torus. We now consider maps of the form $T = T_{A,b}$ given by $T\omega = A\omega + b$ where $A \in \SL(d,\Z)$ and $b \in \T^d$.  Note that Lebesgue measure is $T$-invariant for all $A \in \SL(d,\Z)$ and $b \in \T^d$. These examples will be addressed in detail with proofs in Section~\ref{sec:affine}.
\begin{enumerate}
\item If $\mu$ is a $T_{A,b}$-ergodic measure with $\supp \mu= \T^d$, then
\begin{equation}
\schwartzmanGroup(\T^d,T_{A,b}, \mu) 
= \set{n + \langle k, b \rangle : n \in \Z, \ k \in \Z^d \cap \ker(I-A^*)}.
\end{equation}
\item If $A$ does not have any root of unity as an eigenvalue, $b=0$, and $\mu$ is Lebesgue measure, then $(\T^d,T_{A,0},\mu)$ is ergodic \cite[Corollary~1.10.1]{Walters1982:ErgTh} and 
\begin{equation}
\schwartzmanGroup(\T^d,T_{A,0},\mu) = \Z.
\end{equation}
As mentioned in Corollary~\ref{coro:A=Z}, the conclusion $\schwartzmanGroup(\Omega,T,\mu) = \Z$ implies that $\Sigma_{\mu,f}$ is an interval for every $f \in C(\Omega,\R)$. The \emph{cat map}, defined by $T_{\rm cat}= T_{A,0}$, where
 \begin{equation}
 A = \begin{bmatrix}2 & 1 \\ 1 & 1 \end{bmatrix},
 \end{equation}
 supplies a concrete example of a transformation with the property that the underlying linear transformation has no root of unity as an eigenvalue.

\item The \emph{skew-shift} is defined by taking $d=2$,
\begin{equation}
A = \begin{bmatrix} 1 & 0 \\ 1 & 1 \end{bmatrix}, \quad b =\begin{bmatrix}
\alpha \\ 0
\end{bmatrix}
\end{equation} 
with $\alpha$ irrational. In this case, it is known that $(\T^2,T_{A,b})$ is strictly ergodic \cite{Furst1981Porter}. One has
\begin{equation}
\schwartzmanGroup(\T^2,T_{A,b}) = \Z+\alpha\Z.
\end{equation}
\end{enumerate}
\end{example}

\subsection{Deriving or Preventing Cantor Spectrum} \label{subsec:deriveCantor}

A very important application of the gap labelling theorem involves the ability to show the presence or absence of Cantor-type structures in the spectrum in concrete settings. This is motivated by the discovery in the 1980's that the spectra of Schr\"odinger operators with almost-periodic potentials can be Cantor sets, that is, they can be nowhere dense while not having any isolated points. In fact, an ergodicity argument rules out the existence of isolated points \cite{Pastur1980CMP}, so that the challenge in establishing a Cantor-type structure revolves around finding a proof that the gaps of the spectrum are dense. In other words, any real number can be approximated by sequence of elements of the complement of the spectrum in $\R$.

Gap labelling enters this discussion as follows. As the set of growth points of the integrated density of states coincides with the almost sure spectrum, gaps can be dense only if the possible gap labels are dense. Thus, the first application of the gap labelling theorem is the general criterion that the non-denseness of gap labels implies the absence of a Cantor spectrum. Thus, it is of interest to identify settings where the range of the Schwartzman homomorphism is not dense in $[0,1]$. We have seen some examples above and will provide more detailed explanations below.

If on the other hand the range of the Schwartzman homomorphism is dense in $[0,1]$, this does not by itself imply that the spectrum is a Cantor set, as one does not know which elements of $\schwartzmanGroup \cap [0,1]$ actually appear as gap labels corresponding to open gaps. In other words, the denseness of the range of the Schwartzman homomorphism in $[0,1]$ suggests that Cantor spectrum is in principle possible, but does not establish it. The gap labels that do appear depend on $f$. Thus, a natural goal is to show that for a dense subset of the dense set of all possible labels, there are $f$'s for which the labels in this subset all occur on gaps of $\Sigma_{\mu,f}$. In typical applications, one actually works with the full set of labels, which is natural given the reasoning employed. Namely, one usually fixes an potential gap label and considers the set of $f$'s for which $\Sigma_{\mu,f}$ has a gap with that label. If one is able to show that this set of $f$'s is open and dense in a suitable function space, then by Baire's theorem, Cantor spectrum is generic in that function space. As this argument only uses that the set of labels one intersects over is countable, one may as well intersect over all possible labels, that is, the option to intersect over a countable dense subset does not give any stronger conclusion. We refer the reader to \cite{Eliasson1992}, \cite{PS06}, and \cite{ABD12} for results obtained along the lines just discussed. We also mention that the somewhat involved approach of \cite{ABD12} can be simplified significantly if one is only interested in the conclusion, generic Cantor spectrum, rather than the sufficient condition for it in terms of gap labelling; see \cite{ABD09}.

\subsection{Gap Labelling via K-Theory} \label{subsec:ktheory}
We want to note that there is a beautiful approach to gap labeling, which is in fact substantially more general in that it is not restricted to operators in one space dimension and which also can handle more general covariant families of operators in one dimension. This alternative approach is based on K-theory of C$^*$-algebras and it was developed by Bellissard and coworkers.

Let us describe this approach briefly. For proofs, more details, and further discussion, we refer the reader to \cite{Bel1986, Bel1990, Bel1992b, Bel2003, BelBovGhe1992, KellZois2005JPA}.

A C$^*$-algebra $\mathscr{A}$ is a Banach algebra together with an involution $x \mapsto x^*$ satisfying, among other things, $\|x^*x\| = \|x\|^2$. An element $p \in \mathscr{A}$ is an orthogonal projection (henceforth simply a projection) if $p^2 = p^* = p$. Denote by $\mathscr{P}(\mathscr{A})$ the set of all projections.

Let us assume for simplicity that $\mathscr{A}$ is a unital (i.e., it has a multiplicative identity) C$^*$-algebra. For each $n \in \bbN$, the set $M_n(\mathscr{A})$ of $n \times n$ matrices with entries from $\mathscr{A}$ is a C$^*$-algebra as well with the natural $^*$ operation, and we set
\begin{equation}
\mathscr{P}_n(\mathscr{A}) := \mathscr{P}(M_n(\mathscr{A})), \quad \mathscr{P}_\infty(\mathscr{A}) := \bigcup_{n \in \bbN} \mathscr{P}_n(\mathscr{A}).
\end{equation}
On $\mathscr{P}_\infty(\mathscr{A})$, the relation $\sim_0$ is defined as follows: if $p \in \mathscr{P}_n(\mathscr{A})$ and $q \in \mathscr{P}_m(\mathscr{A})$, then $p \sim_0 q$ if and only if there is $v \in M_{m,n}(\mathscr{A})$ (the $m \times n$ matrices with entries from $\mathscr{A}$) with $p = v^*v$ and $q = vv^*$. It turns out that $\sim_0$ is an equivalence relation and we set
\begin{equation}
K_0^+(\mathscr{A}) := \mathscr{P}_\infty(\mathscr{A}) / \sim_0.
\end{equation}

We also define $\oplus : \mathscr{P}_\infty(\mathscr{A}) \times \mathscr{P}_\infty(\mathscr{A}) \to \mathscr{P}_\infty(\mathscr{A})$ by
\begin{equation}
p \oplus q := \mathrm{diag}(p,q) = \begin{bmatrix} p & 0 \\ 0 & q \end{bmatrix}.
\end{equation}
The operation $\oplus$ on $\mathscr{P}_\infty(\mathscr{A})$ descends to an operation on $K_0^+(\mathscr{A})$ and makes it an abelian semigroup. A standard (Grothendieck) construction then associates an abelian group $K_0(\mathscr{A})$ whose elements are equivalence classes of formal differences of elements of the semigroup, along with a natural extension of the group operation.

The next step is to associate a C$^*$-algebra with an ergodic family $\{ H_\omega \}_{\omega \in \Omega}$ of Schr\"odinger operators. Let us consider the case where $\Omega \subseteq I^\Z$ is closed and shift-invariant for some compact interval $I$ in $\R$ and $T$ is the shift on this sequence space. We can arrive at this setting by pushing forward under $\omega \mapsto V_\omega$ and using the evaluation at zero as the  sampling function. The associated C$^*$ algebra is then given by the crossed product
\begin{equation}
\mathscr{A} := C^*(\Omega, T) := C(\Omega) \rtimes_T \Z
\end{equation}
of $C(\Omega)$ by the $\Z$ action defined by $T$, that is, the closure of $C_c(\Omega \times \Z)$ with respect to the C$^*$ norm
\begin{equation}
\|g\| := \sup_{\omega \in \Omega} \|\pi_\omega(g)\|,
\end{equation}
where $\pi_\omega$ is the family of $^*$-representations on $\ell^2(\Z)$ defined by
\begin{equation}
[\pi_\omega(g) \psi](n) := \sum_{m \in \Z} g(T^{-n} \omega, n) \psi(m+n), \quad g \in C_c(\Omega \times \Z).
\end{equation}

An ergodic measure $\mu$ on $\Omega$ (again, we would need to push forward under $\omega \mapsto V_\omega$ if we start from the abstract setting) defines a trace $\tau_\mu$ on $C^*(\Omega,T)$ in the following way:
\begin{equation}
\tau_\mu(g) = \int_\Omega g(\omega,0) \, d\mu(\omega), \quad g \in C_c(\Omega \times \Z).
\end{equation}

It can then be shown that there is an induced group homomorphism
\begin{equation}
\tau_{\mu}^* : K_0(\mathscr{A}) \to \R
\end{equation}
such that for every projection $p \in \mathscr{A}$, we have $\tau_\mu(p) = \tau_{\mu}^*([p])$, where $[p]$ is the class of $p$ in $K_0(\mathscr{A})$.

\begin{theorem}[Gap labeling via K-theory]
Consider an ergodic family $\{ H_\omega \}_{\omega \in \Omega}$ of Schr\"odinger operators and let $\mathscr{A}$ be the associated C$^*$-algebra. Then for each gap of $\Sigma$, the value the integrated density of states takes inside it belongs to the countable set of real numbers given by $\tau_{\mu}^*(K_0(\mathscr{A})) \cap [0,1]$.
\end{theorem}

In view of this discussion, it is natural to ask whether the two approaches yield identical results. That is:

\begin{question}
Given an ergodic topological dynamical system $(\Omega,T,\mu)$, is it true that 
\begin{equation}
\tau_{\mu}^*(K_0(\mathscr{A})) = \schwartzmanGroup(\Omega,T,\mu)?
\end{equation}
\end{question}

This is addressed in higher dimensions for minimal $\Z^d$ actions on totally disconnected spaces and repetitive Delone sets in \cite{BellissardBenedettiGambaudo2006CMP, BKL01, BenameurOyonoOyono2001, KaminkerPutnam2003MMJ, vE94}.\footnote{To be more precise, while $\schwartzmanGroup$ does not have an immediate analogue in higher dimensions in general, in the specific setting considered in these papers, $\schwartzmanGroup$ can be identified as a set, compare \eqref{e.subshiftSchwartzman}, that does have a natural higher-dimensional counterpart.}  See also \cite{GHK13} for relevant work. When specialized to dimension one as in the present work, one essentially ends up with minimal subshifts as in Section~\ref{sec:subshift}.

Given that there is a more powerful and more general approach to gap-labelling (which applies both to operators in higher dimensions and to more general operators in dimension one), we should reflect on why we felt compelled to elucidate the current perspective. Indeed, it should be understood as a pedagogical and aesthetic choice: presenting in this fashion allows us to give a self-contained presentation modulo tools and techniques  from linear algebra, measure theory, and a few elementary aspects of ergodic theory.

\subsection{Organization} The remainder of the paper is divided into two parts: general theory and examples. The first part of the paper builds up the necessary machinery and proves the one-dimensional version of the gap-labelling theorem formulated in Theorem~\ref{t:gablabel}. We give a brief review of the oscillation theorem for one-dimensional Schr\"odinger operators in Section~\ref{sec:osc}, define and recall basic properties of the integrated density of states in Section~\ref{sec:IDS}, and then recall basic facts about flows and the definition of the Schwartzman homomorphism in Section~\ref{sec:flow}. This is all put together in Section~\ref{sec:gaplabel} to prove Theorem~\ref{t:gablabel}. The second part of the paper then deals with computing $\schwartzmanGroup(\Omega,T,\mu)$ for examples of interest. We deal with almost-periodic potentials in Section~\ref{sec:ap}, subshift and random potentials in Section~\ref{sec:subshift}, and operators generated by general affine toral homeomorphisms in Section~\ref{sec:affine}.

\subsection*{Acknowledgements} We are grateful to Michael Baake, Franz G\"{a}hler, Svetlana Jitomirskaya, Johannes Kellendonk, Andy Putman, Lorenzo Sadun, and Hiro Lee Tanaka for helpful conversations and comments. We also want to thank the American Institute of Mathematics for hospitality and support during a January 2022 visit, during which part of this work was completed.

\section{Oscillation Theory} \label{sec:osc}
Let us briefly leave the ergodic setting and consider a general potential $V:\Z \to \R$, and the associated difference equation
\begin{equation} \label{eq:deve:osc}
u(n-1) + u(n+1) + V(n)u(n) = Eu(n),
\quad
n \in \Z.
\end{equation}
We will seek to connect oscillation of solutions of \eqref{eq:deve:osc} with eigenvalue counting functions of Dirichlet truncations of $H$; for $N \in \N$, define $H_N \in \R^{N \times N}$ by
\begin{equation}
H_N
= \begin{bmatrix} V(1) & 1 \\ 1 & V(2) & 1 \\
& \ddots & \ddots & \ddots \\
&& 1 & V(N-1) & 1 \\
&&& 1 & V(N) \end{bmatrix}.
\end{equation}

Let $u_E:\Z \to \R$ denote the sequence defined by
\begin{equation} u_E \text{ solves \eqref{eq:deve:osc} with } u_E(0) = 0 \text{ and } u_E(1)=1.
\end{equation}
The main theorem of the present section states that the number of sign flips of $u_E$ on the interval $[1,N+1]$ coincides with the number of eigenvalues of $H_N$ that exceed $E$. 

Oscillation and comparison theorems for linear second-order differential operators date back to Sturm \cite{Sturm1936a,Sturm1936b}. To the best of our knowledge, the earliest appearance of discrete versions of these results is due to B\^ocher \cite{Bocher1897/98}. The approach to oscillation theory here is inspired by Simon \cite{Simon2005:OscTh}. See also Teschl \cite[Chapter~4]{Teschl2000:Jacobi} and Zettl \cite{Zettl2005}.

In the present setting, one may readily derive the main results by noting that $u_E(n)$ (as a function of energy) may be realized as the characteristic polynomial of a suitable Dirichlet truncation, and as such a sign flip of $u_E$ between $n$ and $n+1$ can be formulated in terms of eigenvalue counts. More precisely:
\begin{prop} \label{prop:dirich:soln:det}
Let $V: \Z \to \R$ be given. For any $E \in \C$ and any $n \in \N$,
\begin{equation}
u_E(n+1) = \det(E-H_n).
\end{equation}
\end{prop}

\begin{proof}
Both sides are readily seen to be monic polynomials of degree $n$ in $E$ that vanish whenever $E \in \sigma(H_n)$. The reader can check that every eigenvalue of $H_n$ is simple, so the result follows.
\end{proof}

The discussion of sign flips becomes slightly more delicate when $E$ is a Dirichlet eigenvalue, i.e., an eigenvalue of some $H_n$.  In particular, for such $E$, we have $u_E(n+1) = 0$ and
\[
u_E(n+2) = - u_E(n),
\]
so we should only count this as a single sign flip. To accomplish this, linearly interpolate $u_E$ to obtain a continuous real function $u_E(x)$, $x \in \R$.  For each $N \in \N$, let $F_N(E)$ denote the number of zeros of $u_E(x)$ in the interval $(1,N+1)$.  Equivalently,
\begin{equation}
F_N(E)
=
\# \{ 1 \leq j \leq N : \sgn( u_E(j) ) \neq \sgn (u_E(j+1)) \}
\end{equation}
whenever $E \not\in \sigma(H_N)$ and
\begin{equation}
F_N(E)
=
\# \{ 1 \leq j \leq N-1 : \sgn( u_E(j) ) \neq \sgn (u_E(j+1)) \}
\end{equation}
for $E \in \sigma(H_N)$, where we adopt the convention $\sgn(0) = 1$ to avoid over-counting sign flips at the Dirichlet eigenvalues.

\begin{theorem}[Oscillation Theorem] \label{t:sturm:osc}
Fix a potential $V: \Z \to \R$. For all $N \in \N$ and $E \in \R$, $F_N(E)$ is equal to the number of eigenvalues of $H_N$ that exceed $E$. That is,
\begin{equation} \label{e.ids.signflips}
\# \left( \sigma(H_N) \cap (E,\infty) \right)
=
F_N(E)
\text{ for all } E \in \R.
\end{equation}
\end{theorem}

The following lemma will be useful.

\begin{lemma}[Interlacing Lemma] \label{l.cutoff.interlace}
For every $N \geq 2$, the eigenvalues of $H_{N-1}$ strictly interlace those of $H_{N}$.  More precisely, if $ E_{N-1}^1 < \cdots < E_{N-1}^{N-1}$ and $ E_N^1 < \cdots < E_N^N$ denote the eigenvalues of $H_{N-1}$ and $H_N$, respectively, then
\begin{equation} \label{e.cutoff.interlace1}
E_N^j
<
E_{N-1}^j
<
E_N^{j+1}
\end{equation}
for all $1 \leq j \leq N-1$.
\end{lemma}

\begin{proof}
The reader can check that all eigenvalues of $H_N$ and $H_{N-1}$ are simple, so it suffices to show that for any $1 \leq j \leq N-1$, there exists an eigenvalue $\beta$ of $H_{N-1}$ with $E_N^j < \beta < E_N^{j+1}$.  To that end, let $v_1,\ldots,v_N$ be the (orthonormal) eigenvectors of $H_N$, chosen so that
$$
H_N v_j = E_N^j v_j,
$$
and define a meromorphic function $h$ by
\begin{equation}
 h(z) = \frac{\det(z-H_{N-1})}{\det(z-H_{N})}.
\end{equation}
By Cramer's rule, $h(z) =  \left\langle \delta_N, (z-H_{N})^{-1} \delta_N \right\rangle$ for all $z \notin \sigma(H_{N})$.  If we expand $\delta_N$ in the basis $\{ v_1,\ldots, v_N \}$, we see that
\begin{align*}
h(z)
& =
\left\langle \delta_N, (z-H_{N})^{-1} \delta_N \right\rangle \\
& =
\left\langle \sum_{j=1}^N \langle v_j,\delta_N \rangle v_j, \sum_{j=1}^N \frac{1}{z - E_N^j} \langle v_j,\delta_N \rangle v_j \right\rangle\\
& =
\sum_{j=1}^N \frac{|\langle v_j,\delta_N \rangle|^2}{z - E_N^j}
\end{align*}
for all $z \notin \sigma(H_N)$. One can readily check that $\langle v_j,  \delta_N \rangle \neq 0$.  Thus, for each $1 \leq j \leq N-1$, we have
\[
\lim_{x \downarrow E_N^j} h(x) = + \infty
\text{ and }
\lim_{x \uparrow E_N^{j+1}} h(x) = - \infty,
\]
so $h$ vanishes somewhere in the open interval $\left( E_N^j , E_N^{j+1} \right)$. It follows that $H_{N-1}$ has an eigenvalue strictly between $E_N^j$ and $E_N^{j+1}$, as desired.
\end{proof}

With the interlacing lemma in hand, we are ready to prove the desired oscillation theorem.

\begin{proof}[Proof of Theorem~\ref{t:sturm:osc}]
Let $B_N(E)$ denote the number of eigenvalues of $H_N$ that exceed $E$, i.e., the left-hand side of \eqref{e.ids.signflips}. In the case $N = 1$, one has
\begin{align*}
E < E_1^1
& \iff
E < V(1) \\
& \iff
\sgn(E - V(1)) = -1.
\end{align*}
Since $u_E(1) = 1$ and $u_E(2) = E-V(1)$, $ B_1(E) = F_1(E) $ for every $E \neq V(1) $.  If $ E = E_1^1 = V(1) $, then one has $B_1(E) = F_1(E) = 0$.
\newline

Inductively, assume that $B_N \equiv F_N$ for some $N \in \N$, and consider $E \in \R$. Adopting the conventions $E_N^0 = -\infty$ and $ E_N^{N+1} = +\infty $, choose $0 \leq k \leq N $ so that \[E_N^k \leq E < E_N^{k+1}.\] By definition of $B_N$, one has $B_N(E) = N-k$, and then $F_N(E) = N-k$ by the inductive hypothesis. From here, we analyze a few cases separately.
\newline

\begin{case} \textbf{\boldmath$E = E_N^k$.}  It follows that $u_E(N+1) = 0$. Consequently, exactly one of $u_E(N)$, $u_E(N+2,E)$ is positive and the other must be negative. In particular, we have
\[
F_{N+1}(E) = F_N(E) + 1 = N-k+1.
\]
On the other hand, the interlacing lemma (Lemma~\ref{l.cutoff.interlace}) implies that $E_{N+1}^k < E < E_{N+1}^{k+1}$, which yields
\[
B_{N+1}(E) = N-k+1 = F_{N+1}(E).
\]\end{case}

\begin{case} \textbf{{\boldmath$E_N^k < E < E_{N+1}^{k+1}$.}}  First, notice that this gives $B_{N+1}(E) = B_N(E)+1$, as before. Then, since $E$ is neither an eigenvalue of $H_N$ nor of $H_{N+1}$, Proposition~\ref{prop:dirich:soln:det} yields
\begin{align*}
\sgn(u_E(N+2))
& =
\sgn(\det(E - H_{N+1})) \\
& =
(-1)^{B_{N+1}(E)} \\
& \neq
(-1)^{B_N(E)} \\
& =
\sgn(\det(E - H_N)) \\
& =
\sgn(u_E(N+1),
\end{align*}
and hence $F_{N+1}(E) = F_N(E) + 1 = B_{N+1}(E)$.
\end{case}

%\begin{case} \textbf{{\boldmath$E =  E_{N+1}^{k+1}$.}} Exercise. %By interlacing, this implies $B_{N+1}(E) = B_N(E) = N-k$.  $F_N(E) = F_{N+1}(E)$ comes from the definitions, since $E \in \sigma(H_{N+1})$.
%\end{case}

%\begin{case} \textbf{{\boldmath$E_{N+1}^{k+1} < E < E_N^{k+1}$.}} Exercise. %Arguing as in Case~2, we see that $B_{N+1}(E)=B_N(E) $ and $F_{N+1}(E) = F_N(E)$.
%\end{case}

By the interlacing lemma, the remaining cases are $E =  E_{N+1}^{k+1}$ or $E_{N+1}^{k+1} < E < E_N^{k+1}$, which  are similar and left to the reader. \end{proof}

\section{The Integrated Density of States} \label{sec:IDS}

In this section we discuss two closely related fundamental objects associated with an ergodic family of Schr\"odinger operators $\{ H_\omega \}_{\omega \in \Omega}$: the density of states measure and the integrated density of states.

\begin{definition} \label{def:IDSgen:IDSandDOSM}
The \emph{density of states measure} (DOSM) is the measure $\kappa$ defined by
\begin{equation}\label{f.idsdef}
\int \! g \, d\kappa
=
\int_\Omega  \langle \delta_0 , g(H_\omega) \delta_0 \rangle  \, d\mu(\omega)
\end{equation}
for bounded measurable $g$. The function $k$ defined by
\begin{equation}
k(E) = \int \chi_{(-\infty,E]}  \, d\kappa
\end{equation}
is called the \textit{integrated density of states} (IDS).
\end{definition}

Note that by definition, $\kappa$ is the $\mu$-average of the spectral measure  corresponding to the pair $(H_\omega,\delta_0)$.

\medskip

We first establish one of the central properties of the IDS/DOSM, namely its relation with the almost sure spectrum. 

\begin{theorem}\label{t.as83}
The almost sure spectrum is given by the points of increase of $k$. Equivalently, $\Sigma = \supp \kappa$.
\end{theorem}

\begin{proof}
We begin with the inclusion $\supp \kappa \subseteq \Sigma$. If $E_0 \not\in \Sigma$, there is an open interval $I$ containing $E_0$ with $I \cap \Sigma = \emptyset$. For $\mu$-almost every $\omega$, we have $\Sigma = \sigma(H_\omega)$ and hence $\chi_I(H_\omega) = 0$. It follows that
$$
\int  \! \chi_I \, d\kappa
=
\int_\Omega  \langle \delta_0 , \chi_I(H_\omega) \delta_0 \rangle  \, d\mu(\omega) = 0.
$$
Thus, $E_0 \not\in \supp \kappa$.

Let us now prove $\supp \kappa \supseteq \Sigma$: If $E_0 \not\in \supp \kappa$, there is an open interval $I$ containing $E_0$ such that $I \cap \supp \kappa = \emptyset$. But this implies
\begin{align*}
0=\int \!  \chi_I \, d\kappa
& =
\int_\Omega \langle \delta_0 , \chi_I(H_\omega) \delta_0 \rangle  \, d\mu(\omega) \\
& =
\int_\Omega \langle \delta_0 , \chi_I(H_{T^n \omega}) \delta_0 \rangle \, d\mu(\omega) \\
& =
\int_\Omega \langle \delta_n , \chi_I(H_\omega) \delta_n \rangle \, d\mu(\omega),
\end{align*}
where the second step follows by $T$-invariance of $\mu$, and the third follows from the covariance relation $H_{T\omega} = UH_\omega U^*$, where $U$ is the unitary shift $\delta_n \mapsto \delta_{n-1}$. Thus, for $\mu$-almost every $\omega$, we have
\[
\dim \Ran \, \chi_I(H_\omega)
=
\tr \chi_I(H_\omega)
=
0.
\]
Consequently, $E_0 \not\in \Sigma$.
\end{proof}

We now connect the IDS and density of states measure to properties of finite cutoffs of $H_\omega$. To that end, we will show that the density of states measure can also be viewed as the $\mu$-a.e.\ weak limit of normalized traces of truncations of $H_\omega$.

\begin{definition}
For $\omega \in \Omega$ and $N \ge 1$, define the measure $\kappa_{\omega,N}$ by
\begin{equation} \label{eq:IDSgen:dkOmegaNdef}
\int  \! g \, d\kappa_{\omega,N}
=
\frac{1}{N} \, \tr (g(H_\omega) \chi_{[1,N]})
\end{equation}
for bounded measurable $g$.
\end{definition}

\begin{lemma}\label{kapplem}
For every bounded, measurable function $g$, there exists a set $\Omega_g \subseteq \Omega$ of full $\mu$-measure such that
\begin{equation}\label{kapprox}
\lim_{N \to \infty} \int \! g \, d\kappa_{\omega,N}
=
\int \!  g \, d\kappa
\end{equation}
for every $\omega \in \Omega_g$.
\end{lemma}

\begin{proof}
Fix a bounded measurable function $g$ and define
\begin{equation}
\overline{g}(\omega) = \langle \delta_0 , g(H_\omega) \delta_0 \rangle,
\quad \omega \in \Omega.
\end{equation}
Since $g$ is bounded and measurable, $\overline{g} \in L^\infty(\Omega,\mu) \subseteq L^1(\Omega,\mu)$. Therefore,
\begin{align*}
\int \!  g \, d\kappa_{\omega,N} 
& = \frac{1}{N} \, \tr (g(H_\omega) \chi_{[1,N]}) \\
& = \frac{1}{N} \, \sum_{n = 1}^{N} \langle \delta_n , g(H_\omega) \delta_n \rangle \\
& = \frac{1}{N} \, \sum_{n = 1}^{N} \langle \delta_0 , g(H_{T^n \omega}) \delta_0 \rangle \\
& = \frac{1}{N} \, \sum_{n = 1}^{N} \overline{g} (T^n \omega) \\
& \to \, \int_\Omega \!\overline{g} \, d\mu \\
& = \int  \! g \, d\kappa.
\end{align*}
In the penultimate step, we applied the Birkhoff ergodic theorem, which works for $\mu$-almost every $\omega \in \Omega$.
\end{proof}

\begin{coro} \label{coro:dosweakconv}
For $\mu$-almost every $\omega$, $\kappa_{\omega,N}$ converges weakly to $\kappa$ as $N \to \infty$. That is, there exists a set $\Omega_*$ of full $\mu$-measure such that \eqref{kapprox} holds for all $\omega \in \Omega_*$ and all $g \in C(\R)$.
\end{coro}

\begin{proof}
Since $\|H_\omega\| \le 2 +\|f\|_\infty =: M$ for a.e.\ $\omega$, it follows that the interval $J = [-M,M]$ supports $dk_{\omega,N}$ for a.e.\ $\omega$ and $N$. Consequently, it suffices to show that there is a full-measure set $\Omega_* \subseteq \Omega$ such that \eqref{kapprox} holds for all $g \in C(J,\R)$. Since $J$ is compact, $C(J)$ enjoys a countable dense subset, $\mathcal D$. The statement of the corollary then follows with
\begin{equation}
\Omega_*
=
\bigcap_{g \in \mathcal D} \Omega_g,
\end{equation}
by an $\varepsilon/3$ argument (where $\Omega_g$ is as defined in Lemma~\ref{kapplem}).
\end{proof}

\begin{theorem} \label{t:idscont}
The density of states measure and the integrated density of states are both continuous.
\end{theorem}

\begin{proof}
Fix $E_0 \in \R$ and suppose $g_n$ are continuous compactly supported functions with $0 \le g_n \le 1$, $g_n(E_0) = 1$, and $g_n(E) \downarrow 0$ for $E \not= E_0$. Then, by dominated convergence,
\begin{equation}\label{fnlimpos}
\int  \! g_n \, d\kappa \to \kappa(\{E_0\}).
\end{equation}
Moreover, $\langle \delta_0 , g_n(H_\omega) \delta_0 \rangle \to \langle \delta_0 , \chisub{\{E_0\}}(H_\omega) \delta_0 \rangle$ for all $\omega$, and hence, again by dominated convergence,
\begin{equation} \label{eq:IDSgeneral:idscont2}
\int  \! g_n \, d\kappa 
= \int_\Omega \langle \delta_0 , g_n(H_\omega) \delta_0 \rangle \, d\mu(\omega) 
\to \int_\Omega \langle \delta_0 ,
\chisub{\{E_0\}}(H_\omega) \delta_0 \rangle \, d\mu(\omega).
\end{equation}
By a standard argument using ergodicity, $\tr(\chisub{\{E_0\}}(H_\omega))$ is almost-surely constant, and the almost-sure value lies in $\{0,\infty\}$. Since $\tr(\chisub{\{E_0\}}(H_\omega)) \leq 1$ for every\footnote{The trace is at most two since the solution space of $H_\omega \psi = E_0\psi$ is two-dimensional; this bound already suffices for this argument. One can reduce the two to one by using constancy of the Wronskian} $\omega$, we have $\chisub{\{E_0\}}(H_\omega) = 0$ for $\mu$-a.e.\ $\omega$. Combining this observation with \eqref{fnlimpos} and \eqref{eq:IDSgeneral:idscont2}, we observe $\kappa(\{E_0\}) = 0$, which implies that $\kappa$ is a continuous measure. Since $k$ is the accumulation function of $\kappa$, continuity of $k$ follows.
\end{proof}

A different approach to the integrated density of states is obtained if we restrict the operator to a finite box, rather than its spectral projection. 
 \begin{definition} \label{def:IDS:tildekdef}
Denote the restriction of $H_\omega$ to $[1,N]$ with Dirichlet boundary conditions by $H_{\omega,N}$, that is,
\begin{equation}
H_{\omega,N}
=
\begin{bmatrix}
V_\omega(1) & 1 &&&\\
1 & V_\omega(2) & 1 &&\\
& \ddots & \ddots  & \ddots & \\
&& 1  & V_\omega(N-1) & 1 \\
&&& 1 & V_\omega(N)
\end{bmatrix}.
\end{equation}
For $\omega \in \Omega$ and $N \ge 1$, define measures $\widetilde{\kappa}_{\omega,N}$ by placing uniformly distributed point masses at the eigenvalues $E^{(1)}_{\omega,N} < \cdots < E^{(N)}_{\omega,N}$ of $H_{\omega,N}$ (recall that every eigenvalue of $H_{\omega,N}$ is simple, as observed in Proposition~ \ref{prop:dirich:soln:det}). That is, we define
\begin{equation}
\int \! g \, d\widetilde{\kappa}_{\omega,N} = \frac{1}{N} \, \sum_{n = 1}^N g\left( E^{(n)}_{\omega,N} \right)
\end{equation}
for  bounded measurable $g$.
\end{definition}

Note that $\{g(E_{\omega,N}^j): 1 \le j \le N\}$ is precisely the spectrum of $g(H_{\omega,N})$, so we can rewrite the definition of $\widetilde{\kappa}_{\omega,N}$ as follows:
\begin{equation}
\int \! g \, d\widetilde{\kappa}_{\omega,N} = \frac{1}{N} \tr(g(H_\omega \chi_{[1,N]})).
\end{equation}
This also makes the connection to $\kappa_{\omega,N}$ more transparent (compare \eqref{eq:IDSgen:dkOmegaNdef}). To be more explicit, both measures involve cutting off and taking a normalized trace. The measure $\kappa_{\omega,N}$ applies $g$ to $H_\omega$ and cuts the result off, while $\widetilde{\kappa}_{\omega,N}$ cuts $H_\omega$ off first and then applies $g$.

The following theorem shows that the two perspectives on the density of states measure indeed are equivalent.

\begin{theorem}\label{t.evrestdefids}
For $\mu$-almost every $\omega \in \Omega$, the measures $\widetilde{\kappa}_{\omega,N}$ converge weakly to $\kappa$ as $N \to \infty$.
\end{theorem}

\begin{proof}
Since all spectra are contained in the interval $[-2-\|f\|_\infty,2+ \|f\|_\infty]$, it suffices to prove, for $\mu$-almost every $\omega \in \Omega$, that the moments of $\widetilde{\kappa}_{\omega,N}$ converge to the moments of $\kappa$ as $N \to \infty$. Given $p \in \Z_+$, notice
\begin{align*}
\int E^p \, d\widetilde{\kappa}_{\omega,N}(E) & = \frac{1}{N} \, \sum_{n = 1}^N \left(E^{(n)}_{\omega,N} \right)^p \\
& =
\frac{1}{N} \, \tr \left( (H_{\omega,N})^p \right) \\
& =
\frac{1}{N} \, \sum_{n = 1}^N \left\langle \delta_n , (H_{\omega,N})^p \, \delta_n \right\rangle \\
& =
\frac{1}{N} \, \left[ \sum_{n = 1}^N \left\langle \delta_n , (H_\omega )^p \, \delta_n \right\rangle + O(1) \right] \\
& =
\frac{1}{N} \, \left[ \sum_{n = 1}^N \left\langle \delta_0 , (H_{T^n \omega} )^p \, \delta_0 \right\rangle + O(1) \right].
\end{align*}
By the Birkhoff ergodic theorem, it follows that
$$
\lim_{N \to \infty} \frac{1}{N} \, \sum_{n = 1}^N \left\langle \delta_0 , (H_{T^n \omega} )^p \, \delta_0 \right\rangle
=
\int_\Omega \langle \delta_0 , (H_\omega )^p \delta_0 \rangle \, d\mu(\omega)
$$
for $\mu$-almost every $\omega \in \Omega$. Thus, for these $\omega$'s,
$$
\lim_{N \to \infty} \int E^p \, d\widetilde{\kappa}_{\omega,N}(E) 
= \int_\Omega \langle \delta_0 , (H_\omega )^p \, \delta_0 \rangle \, d\mu(\omega)
 = \int E^p \, d\kappa(E),
$$
as desired.
\end{proof}

We conclude the present section by showing how to view the integrated density of states as a rotation number by counting sign flips of Dirichlet eigenfunctions using the discrete version of Sturm oscillation theory described in Section~\ref{sec:osc}.

\begin{theorem} \label{t:ids:rotnumber}
If $k$ denotes the integrated density of states of the ergodic family $\{H_\omega\}_{\omega \in \Omega}$, one has
\[
1-k(E)
=
\lim_{N\to \infty} \frac{1}{N} \# \{ 1 \leq j \leq N : \sgn(u_{E,\omega}(j)) \neq \sgn(u_{E,\omega}(j+1)) \}
\]
for $\mu$-almost every $\omega \in \Omega$ and every $E \in \R$, where $u_{E,\omega}$ is the solution to $H_\omega u = Eu$ with $u(0) = 0$ and $u(1) = 1$.
\end{theorem}

\begin{proof}
Since $\widetilde \kappa_{\omega,N}$ converges weakly to $\kappa$ for almost every $\omega \in \Omega$ by Theorem~\ref{t.evrestdefids} and $\kappa$ is continuous (by Theorem~\ref{t:idscont}), we have
\begin{equation}\label{eq:IDSgen:weakconvgoal}
\lim_{N \to\infty} \int \chi_{(-\infty,E]} \, d\widetilde{\kappa}_{\omega,N}
=
\int \chi_{(-\infty,E]} \, d\kappa
\end{equation}
for every $E \in \R$ and almost every $\omega$.
Concretely, let $f_n$ and $g_n$ be continuous functions with $f_n \equiv 1$ on $(-\infty,E-1/n]$, $g_n \equiv 1$ on $(-\infty,E]$, $f_n \equiv 0$ on $[E,\infty)$, $g_n \equiv 0$ on $[E+1/n,\infty)$ (and linearly interpolated in between). Then, for any $\omega$ for which $d\widetilde{k}_{\omega,N}\to dk$ weakly, one obtains
\begin{align*} \int \! f_n \, d\kappa 
& = \lim_{N \to \infty} \int \! f_n \, d\widetilde{\kappa}_{\omega,N} \\
& \leq \liminf_{N \to \infty} \int \! \chi_{_{(-\infty,E]}}\, d\widetilde{\kappa}_{\omega,N} \\
& \leq \limsup_{N \to \infty} \int \! \chi_{_{(-\infty,E]}}\, d\widetilde{\kappa}_{\omega,N}\\
& \leq \lim_{N \to \infty} \int \! g_n \, d\widetilde{\kappa}_{\omega,N} \\
& = \int \! g_n \, d\kappa .\end{align*}
By dominated convergence and continuity of $\kappa$, $\int(g_n-f_n) \, dk \to 0$ as $n \to\infty$, concluding the proof of \eqref{eq:IDSgen:weakconvgoal}.

Consequently, for a.e.\ $\omega$, one has
\begin{align*}
1-k(E)
   = \int \! \chisub{(E,\infty)} \, dk
 & =\lim_{N \to \infty} \int \! \chisub{(E,\infty)} \, d\widetilde{k}_{\omega,N}\\
 & = \lim_{N \to \infty} \frac{1}{N} \# (\sigma(H_{\omega,N}) \cap (E,\infty)),\end{align*}
so the desired conclusion follows from Theorem~\ref{t:sturm:osc}.
\end{proof}

\section{Flows, Suspensions, and the Schwartzman Homomorphism} \label{sec:flow}

We will briefly recall some facts about continuous flows on compact metric spaces. The overall goal is to build up enough machinery to describe the Schwartzman homomorphism, which provides a version of rotation number for continuous maps to the circle originating from a space on which a flow is defined. Later in the paper, we will relate the Schwartzman homomorphism to the rotation number for a cocycle originating from a Schr\"odinger operator, which is in turn related to sign flips of Dirichlet solutions via Theorem~\ref{t:ids:rotnumber} and hence to the IDS for ergodic operators.

\subsection{Basics}

Let $X$ denote a compact metric space. A \emph{flow} on $X$ will mean an action of $\R$ on $X$ by homeomorphisms. That is, for each $t \in \R$, there is a homeomorphism $T^t : X \to X$ so that $T^t \circ T^s = T^{t+s}$ for all $s, t \in \R$ and $(x,t) \mapsto T^t x$ is a continuous map $X \times \R \to X$. If $X$ is in addition equipped with a Borel probability measure $\mu$, we say that the flow is \emph{measure-preserving} if one has $\mu\left(T^t B\right) = \mu(B)$ for every $B \in \mathcal{B} $ and every $t \in \R$. We want to emphasize at the outset that  we assume no regularity of $T$ beyond mere continuity, and indeed no differentiable structure on $X$; in particular, $X$ does not need to be a manifold. In our main applications, $X$ will be the suspension of a suitable discrete-time dynamical system and hence will sometimes come equipped with a local differentiable structure \emph{along the flow}, but we again emphasize that our typical application will be something like the suspension of a translation on a compact group or the suspension of the shift on a subshift over a finite alphabet. 

We say that a Borel probability measure on $X$ is $T$-\emph{ergodic} if it is measure-preserving and every flow-invariant function is almost surely constant. That is to say, if $g$ is a measurable function on $ X $ and $g(T^tx) = g(x)$ for every $t \in \R$, $x \in X $, then there is a constant $g_*$ and a set $X_* \subseteq X$ of full $\mu$-measure so that $ g(x) = g_* $ for every $ x \in X_* $.  Equivalently, we say that the measure preserving flow described by $T^t$ is ergodic if every measurable set $E \subseteq X$ with $T^t E = E$ for every $t$ satisfies $\mu(E) = 0$ or $\mu(X \setminus E) = 0 $.

If $T$ denotes a flow on $X$, then the map $T^1$ is often called the \emph{time-one map}.  Thus, associated to any flow, there is a natural discrete-time dynamical system, namely $(X,T^1)$.

\begin{theorem} \label{t:birkhoff:flow}
Suppose $(X,\mu)$ is a probability space and $T$ is a flow on $X$ that is ergodic with respect to $\mu$. Then
\begin{equation} \label{eq:birkhoff:flow}
\lim_{t \to \infty} \frac{1}{t} \int_0^t \! g(T^s x) \, ds
=
\int_X g \, d\mu
\end{equation}
for every $g \in L^1(\mu)$ and $\mu$-almost every $x \in X$.
\end{theorem}

Theorem~\ref{t:birkhoff:flow} is well known and can be deduced from the discrete-time version of the ergodic theorem; compare \cite{BerLeiMor2012}.

We will eventually want to apply the Schwartzman homomorphism to an invariant section derived from the (discrete) cocycle dynamics associated with $\{H_{f,\omega}\}$. In order to do that, we need an appropriate analog of a continuous cocycle over a flow.

\begin{definition}
A continuous  $\SL(2,\R)$ \emph{cocycle} over the flow $T$ is defined to be a continuous map $\Phi:X \times \R \to \SL(2,\R)$ with the properties
\begin{align}
\Phi(x,0) & = I \\
\Phi(x,s+t) & = \Phi(T^s x, t) \cdot \Phi(x, s)
\end{align}
for all $x \in X$ and $s,t \in \R$. We will denote $ \Phi(x,t) = \Phi^t(x)$.
\end{definition}

A cocycle induces flows on $X \times \R^2$ and $X \times \R\PP^1$ via $T^t (x,v) = (T^t x, \Phi^t(x)v)$. One has the following notions of uniform hyperbolicity for continuous cocycles.

\begin{definition}
Let $\Phi$ be a continuous cocycle. We say that $\Phi$ exhibits \emph{uniform exponential growth} if there are constants, $C>0$, $\lambda>1$ with the property that
\begin{equation}
\| \Phi^t(x) \|
\geq
C \lambda^{|t|}
\end{equation}
for all $t \in \R$ and $x \in X$.

 We say that $\Phi$ admits an \emph{invariant exponential splitting} if there exist constants $c > 0, L > 1$ and continuous maps $\Lambda^s,\Lambda^u : X \to \R \PP^1$ such that the following statements hold.
\begin{itemize}
\item[{\rm (a)}] (Invariance) For all $x \in X$, and $t \in \R$, one has
\begin{equation}
\Phi^t(x) \Lambda^s (x) = \Lambda^s(T^tx),
\quad
\Phi^t(x)\Lambda^u (x) = \Lambda^u(T^t x).
\end{equation}

\item[{\rm (b) }] (Contraction) For all $t > 0$, $x\in X$, $v_s \in \Lambda^s(x)$, and $ v_u \in \Lambda^u(x) $, one has
\begin{equation}
\| \Phi^t   (x) v_s \| \leq c L^{-t} \|v_s\|,
\quad
\| \Phi^{-t}(x) v_u \| \leq c L^{-t} \|v_u\|.
\end{equation}
\end{itemize}

We say that $\Phi$ enjoys a \emph{bounded orbit} if there exist $x \in X$ and $v \in \X^1 $ so that
\begin{equation}
\sup_{t \in \R} \| \Phi^t(x) v \| < \infty.
\end{equation}
\end{definition}

\begin{theorem} \label{t:uh:splitting:continuous}
Suppose $\Phi:X \times \R \to \SL(2,\R)$ is a continuous cocycle over a continuous flow $T$ on a compact metric space $X$. Then the following are equivalent.
\begin{enumerate}
\item[{\rm(a)}] $\Phi$ exhibits uniform exponential growth.
\item[{\rm(b)}] $\Phi$ admits an invariant exponential splitting.
\item[{\rm(c)}] $\Phi$ does not enjoy a bounded orbit.
\end{enumerate}
\end{theorem}

\begin{proof}
Consider the discrete cocycle $(T^1,A)$ associated with the time-one map, that is, the skew product $(T^1,A):X \times \R^2 \to X \times \R^2$ given by $(x,v) \mapsto (T^1x,\Phi^1(x)v)$. There are notions of uniform hyperbolicity for such discrete cocycles that are almost identical to those discussed above for continuous cocycles; compare \cite{Zhang2020JST}.

Then, use continuity of $\Phi$ to interpolate the corresponding discrete results characterizing uniform hyperbolicity of $(T^1,A)$ (e.g.\ \cite{DFLY2016DCDS, Zhang2020JST}). Alternatively, one can prove the equivalence of (a), (b), and (c) directly by following the proofs of corresponding discrete-time results and making minor cosmetic changes. The details are left to the reader.
\end{proof}

\subsection{The Suspension of a Dynamical System}

It is easy to produce a discrete-time dynamical system from a flow by passing to the time-one map.  To go in the other direction, one may consider the suspension.  
\begin{definition} \label{def:flow:suspension}
Suppose $(\Omega,T)$ is a topological dynamical system, that is, $\Omega$ is a compact metric space and $T:\Omega \to \Omega$ is a homeomorphism.  Define an action of $\Z$ on $\Omega \times \R$  by
\begin{equation} n \odot (\omega, t) = (T^n \omega, t - n),
\end{equation}
which induces an equivalence relation on $\Omega \times \R$ via
\begin{align}
(\omega,t) \sim (\omega',t')
& \iff
n \odot (\omega,t) = (\omega',t') \text{ for some } n \in \Z \\
& \iff
 (t-t')  \in \Z,
\text{ and }
T^{t-t'} \omega = \omega'.
\end{align}
The \emph{suspension} of the dynamical system $(\Omega,T)$ is the quotient of $\Omega \times \R$ by this equivalence relation, i.e.
\begin{equation}
X(\Omega,T)
:=
(\Omega \times \R) / \! \sim.
\end{equation}
Generally, the homeomorphism and compact space are clear from the context, so we will sometimes simply write $X$.  Naturally, there are homeomorphic copies of $\Omega$ inside $X$, namely the fibers $\Omega_s = \{ [\omega,s] : \omega \in \Omega \}$ for each fixed $s \in \R$, where $[\omega,t]$ denotes the equivalence class of $(\omega,t)$ in $X$. There is a natural $\R$-flow on $X$ defined by translation in the second factor, which we denote by $\oT$. That is, $\oT^t[\omega,s] = [\omega,s+t]$.  For each $\omega \in \Omega$, one has
\begin{equation}
\oT^1 [\omega,s]
=
[\omega,s+1]
=
[T \omega, s].
\end{equation}
In other words, $\oT^1$ maps the fiber $\Omega_s$ to itself and the action of $\oT^1$ on the fiber $\Omega_s$ is the same as the action of $T$ on $\Omega$.
\end{definition}

\begin{remark}
One could equivalently realize $X$ as the quotient $\Omega \times [0,1] /\!\!\sim$ where $ (\omega,1) \sim (T \omega, 0) $. The reader can readily check that the map sending $[\omega,t]\mapsto [T^{\lfloor t \rfloor}\omega,\{t\}]$ is a homeomorphism from the first version of $X$ to this version. Henceforth, we freely work with whichever realization of the suspension is more convenient.
\end{remark}

We will also need the notion of the suspension of a measure $\mu$ on $\Omega$, which is the measure $\overline{\mu}$ defined on $X$ by
\begin{equation}
\int_{X} \! g \, d\overline{\mu}
=
\int_0^1 \! \int_{\Omega} \! g[\omega,s] \, d\mu(\omega) \, ds.
\end{equation}

\begin{lemma}
Fix a topological dynamical system $(\Omega,T)$, a Borel probability measure $\mu$ on $\Omega$, and $X = X(\Omega,T)$.
\begin{enumerate}
\item[{\rm(a)}] If $\mu$ is $T$-invariant, then $\overline{\mu}$ is $\oT$-invariant.
\item[{\rm(b)}] If $\mu$ is $T$-ergodic, then $\overline{\mu}$ is $\oT$-ergodic.
\end{enumerate}
\end{lemma}

\begin{proof}
This follows from direct calculations. We leave the details to the reader.\end{proof}

\subsection{The Schwartzman Homomorphism}
\label{subsec:schwartzman}

Let us return to the setting of general flows to define the Schwartzman homomorphism, which will play a key role in the gap 
labelling theorem. For the remainder of the section, fix a compact metric space $X$, a continuous flow $T$ on $X$, and a $T$-ergodic probability measure $\mu$. Consider $C(X,\T)$, the set of continuous maps $X \to \T$ endowed with the group operation of pointwise addition, i.e.
\begin{equation} \label{e.pi1.group.law}
(\phi_1 + \phi_2)(x)
=
\phi_1(x) + \phi_2(x),
\quad
x \in X.
\end{equation}
Denote the unit interval by $I = [0,1]$. Let us recall some general definitions.
\begin{definition} \label{def:homotopy}
 Given a topological space $\mathcal X$, continuous functions $\phi,\phi': X \to \mathcal X$ are said to be \emph{homotopic} if there is a continuous function $F:X \times I \to \mathcal X$ (called a \emph{homotopy}) so that
\begin{align}
F(x,0) & = \phi(x) \\
F(x,1) & = \phi'(x)
\end{align}
for all $x \in X$. One can confirm that the relation \begin{equation}
\phi \sim \phi' \iff \phi \text{ is homotopic to } \phi'
\end{equation}
 defines an equivalence relation on $C(X,\T)$; we denote by $[\phi]$ the equivalence class of $\phi$ modulo this relation and refer to it as the \emph{homotopy class} of $\phi$. Let $C^\sharp(X,\T)$ denote the set of equivalence classes in $C(X,\T)$ modulo homotopy. 
 \end{definition}
 
 We collect a few significant facts in the following proposition.

\begin{prop} \label{prop:flow:homotopyClassFacts}
Let $X$ be a compact metric space.
\begin{enumerate}
\item[{\rm(a)}] If $\phi_1,\phi_2,\phi_1',\phi_2' \in C(X,\T)$ are such that $\phi_j \sim \phi_j'$ for $j=1,2$, then $(\phi_1+\phi_2) \sim (\phi_1' + \phi_2')$.
\item[{\rm(b)}] The operation
\[[\phi_1]+[\phi_2] = [\phi_1+\phi_2], \quad \phi_1,\phi_2 \in C(X,\T) \]
is well-defined and gives $C^\sharp(X,\T)$ the structure of an abelian group.
\item[{\rm(c)}] Let $d$ denote the standard metric on $\T$, given by \begin{equation}
d(x,y) = \dist(x-y,\Z).
\end{equation}
If $\phi_1,\phi_2\in C(X,\T)$ are such that
\[d(\phi_1(x),\phi_2(x)) < 1/2\]
for all $x$, then $\phi_1 \sim \phi_2$.
\item[{\rm(d)}] $C^\sharp(X,\T)$ is countable.
\end{enumerate}
\end{prop}

\begin{proof}
  If $F_j$ is a homotopy from $\phi_j$ to $\phi_j'$ for $j = 1,2$, then one can check that $F_1+F_2$ is a homotopy between $\phi_1 + \phi_2$ and $\phi_1' + \phi_2'$, proving the assertion from (a). The assertion in (b) follows. We leave the assertion from (c) to the reader. Since $X$ is compact, $C(X,\T)$ is separable. Denoting a countable dense set by $\{\phi_n\}$, we see from (c) that every homotopy class in $C^\sharp(X,\T)$ contains at least one $\phi_n$, which suffices to prove the assertion from (d).
\end{proof}

The following fact about homotopic maps into $\T$ will be useful.

\begin{prop} \label{p:nullhomo:lifting}
Let $X$ be a compact metric space. Then two maps $\phi,\phi' : X \to \T$ are homotopic if and only if there exists a continuous map $\widetilde{\psi} : X \to \R$ so that
\begin{equation} \label{eq:difference.lift}
\phi(x) = \phi'(x) + \pi \left( \widetilde{\psi}(x) \right)
\end{equation}
for all $x \in X$, where $\pi: \R \to \T$ denotes the canonical projection.
\end{prop}

\begin{proof}
If \eqref{eq:difference.lift} holds, then
$$
F(x,t)
=
\phi(x) - \pi \left(t \widetilde{\psi}(x) \right)
$$
defines a homotopy from $\phi$ to $\phi'$.

For the converse, notice that it suffices to deal with the case when $\phi' \equiv 0$. To that end, assume that $F$ is a homotopy from $\phi$ to the constant function $0$. Let $\Gamma$ denote the space of continuous functions $\gamma:I \to \T$ with the property that $\gamma(1) = 0$, topologized with the uniform metric
\begin{equation}
d_\unif(\gamma,\gamma')
=
\sup_{0 \le t \le 1} d(\gamma(t),\gamma'(t)),
\end{equation}
and observe that $\Gamma$ is contractible, path-connected, and locally path-connected. The homotopy $F$ induces a continuous map $f: X \to \Gamma$ by mapping $x \mapsto F(x,\cdot)$. If $\Pi:\Gamma \to \T$ is given by $\Pi(\gamma) = \gamma(0)$, then we have the following commutative diagram:
\begin{equation}
\xymatrix{
& \Gamma \ar[rd]^\Pi & \\
X \ar[ru]^f \ar[rr]^\phi &  & \T.
}
\end{equation}
Since $\Gamma$ is contractible, path connected, and locally path connected, we may use covering space theory (e.g.\ %\cite[Proposition~1.33]{Hatcher2001:AlgTop}
\cite{Hatcher2001:AlgTop}) to lift $\Pi$ to a map $\widetilde \Pi:\Gamma \to \R$ so that the following diagram commutes:
\begin{equation}
\xymatrix{
&& \R \ar[dd]^\pi \\
& \Gamma \ar[ru]^{\widetilde \Pi} \ar[rd]^\Pi & \\
X \ar[ru]^f \ar[rr]^\phi &  &\, \T.
}
\end{equation}
With $\widetilde{\psi} = \widetilde{\Pi} \circ f$, we have the desired result.
\end{proof}

\begin{remark}
The reader should notice that the Proposition~\ref{p:nullhomo:lifting} is not a trivial corollary of basic lifting theorems from geometric topology since we have not assumed that $X$ is path connected, nor have we assumed that it is locally path connected; indeed some of the applications we have in mind will involve suspensions of totally disconnected spaces.
\end{remark}

Given a continuous map $\phi:X \to \T $ and $ x \in X$, we obtain a continuous function $ \phi_x: \R \to \T $ by following the image of $\phi$ along the orbit of $x$, i.e.
\begin{equation}
\phi_x(t)
=
\phi(T^t x).
\end{equation}
We then want to talk about the ``rotation number'' of $\phi$ by tracking the average rate at which $\phi_x(t)$ winds around $\T$. To make this precise, recall from the general covering space theory that, for each $x$, there are continuous functions $\widetilde{\phi}_x : \R \to \R$ that satisfy $\pi \circ \widetilde{\phi}_x = \phi_x $, where $\pi : \R \to \T$ is the canonical projection. Indeed, there is a unique such lift satisfying
\begin{equation}
\widetilde{\phi}_x(0) = a
\end{equation}
for each $a \in \pi^{-1}\{\phi_x(0)\}$. Moreover, if $\widetilde{\psi}$ and $\widetilde{\phi}$ are any two distinct lifts of $\phi_x$, then there is an integer $n$ such that $\widetilde{\phi} \equiv \widetilde{\psi} + n$. We will say that $\phi$ is \emph{differentiable} along the flow if $\widetilde \phi_x$ is a differentiable function for every $x \in X$. One can check that this notion does not depend on the choice of lift.
 We denote the \emph{derivative} of $\phi$ with respect to the flow by $\partial \phi $, i.e.
\begin{equation}
(\partial \phi)(x)
=
\frac{d\widetilde\phi_x}{dt}(0).
\end{equation}
Note for later use that one has for $s \in \R$ and $x \in X$ the relationship
\begin{equation}
(\partial \phi)(T^sx)
=
\frac{d\widetilde\phi_x}{dt}(s).
\end{equation}
Finally, let us say that $\phi$ is $C^1$ along the flow if $\partial \phi$ exists everywhere and is a continuous function $X \to \R$.

In order to measure the average rate of winding of $\phi$ along the orbit of some point $x \in X$, the \emph{rotation number} of $\phi$ along the orbit of $x$ is defined by
\begin{equation} \label{eq:schwarz:rotnumber:def}
\rot(\phi; x)
=
\lim_{t \to \infty} \frac{\widetilde{\phi}_x(t)}{t},
\end{equation}
whenever this limit exists.
The next crucial step is to define a homomorphism $\mathfrak{A} : C(X,\T) \to \R$ by sending a map $\phi \in C(X,\T)$ to its rotation number. There is a slight complication in that the limit defining the rotation number may not exist uniformly and it may depend on the choice of $x$.  We will presently prove that the limit on the right hand side of \eqref{eq:schwarz:rotnumber:def} is defined and constant for a set of full $\mu$-measure (recall that we are assuming that $\mu$ is a $T$-ergodic measure). We then define $\mathfrak{A}= \mathfrak{A}_\mu$ to be this almost-sure value:
\begin{equation} \label{e.gap.label.homom}
\mathfrak{A}_{\mu}(\phi)
=
\int_X \rot(\phi;x)) \, d\mu(x)
=
\lim_{t \to \infty} \frac{\widetilde{\phi}_x(t)}{t},
\quad
\mu\text{-a.e.\ } x.
\end{equation}
To establish the existence of the limit, we will apply a smoothing argument to show that an arbitrary continuous function on $X$ can be approximated by integrals along the flow.  More precisely, we will show that the set of functions that are $C^1$ along the flow is uniformly dense in the space of continuous functions. If $\phi$ is $C^1$ along the flow, we then show that \begin{equation}
\rot(\phi;x) = \int_X \partial \phi\, d\mu \ \text{ a.e.\ } x \in X
\end{equation} by Birkhoff's theorem.

\begin{lemma} \label{l:kakutani:smoothing}
Suppose $T$ is a continuous flow on the compact metric space $X$. Then, the set of functions that are $C^1$ along the flow is dense in $C(X,\T)$ with respect to the uniform topology.
\end{lemma}

The lemma is taken from \cite{Schwarzmann1957Annals}; Schwartzman attributes the statement to Kakutani.

\begin{proof}[Proof of Lemma~\ref{l:kakutani:smoothing}]
It is somewhat more straightforward to work with functions
\begin{equation}
X \to \X^1 := \{z \in \C : |z|=1\}.
\end{equation}
 Naturally, one can define the notion of flow-differentiability for elements of $C(X,\X^1)$, and it suffices to prove that $C^1$  functions from $X$ to $\X^1$ are dense in $C(X,\X^1)$.  To that end, assume given $\phi \in C(X,\X^1)$ and  $\varepsilon > 0$, and define $\psi = \psi^\varepsilon \in C(X,\C) $ by
\begin{equation}
\psi(x)
=
\frac{1}{\varepsilon} \int_0^\varepsilon \! \phi(T^s x) \, ds,
\quad
x \in X.
\end{equation}
One can verify that $\psi^\varepsilon$ converges uniformly to $\phi$ as $\varepsilon \downarrow 0$. In particular, if $\varepsilon$ is small enough, $\psi(x) \neq 0$ for all $x \in X$. For such small $\varepsilon$, %let $\eta = \eta^\varepsilon := P \circ \psi^\varepsilon$, where $P:w \mapsto w/|w|$ is the standard retraction that collapses $\C \setminus \{0\}$ onto $\X^1$. By a direct calculation, $\eta^\varepsilon \in C(X,\X^1)$ and $\eta^\varepsilon$ converges uniformly to $\phi$ as $\varepsilon \downarrow 0$, so it remains to be seen that $\eta$ is $C^1$ along the flow. Fix $x \in X$ and $h \neq 0$.  One has\begin{align*}\frac{\psi(T^h x) - \psi(x)}{h}& =\frac{1}{h\varepsilon} \int_0^\varepsilon \! \left( \phi(T^{s+h} x) - \phi(T^s x) \right) \, ds \\& =\frac{1}{h \varepsilon} \left( \int_h^{\varepsilon + h} \! \phi(T^s x) \, ds - \int_0^\varepsilon \! \phi(T^s x) \, ds \right) \\& =\frac{1}{h \varepsilon} \left( \int_\varepsilon^{\varepsilon + h} \! \phi(T^s x) \, ds - \int_0^h \! \phi(T^s x) \, ds \right).\end{align*}Sending $h$ to zero, we see that $(\partial \psi)(x) = \varepsilon^{-1}\left( \phi(T^\varepsilon x) - \phi(x) \right) $.  In particular, $\psi$, and hence $\eta$, is $C^1$ along the flow.
define $\eta^\varepsilon(x) = \psi^\varepsilon(x)/|\psi^\varepsilon(x)|$. The reader may check that $\eta^\varepsilon$ is $C^1$ and $\eta^\varepsilon \to \psi$ uniformly as $\varepsilon\downarrow 0$.
\end{proof}

\begin{coro} \label{coro:flowc1:homotop}
Any $\phi \in C(X,\T)$ is homotopic to a function that is $C^1$ along the flow.
\end{coro}

\begin{proof}
If $\phi$ and $\phi'$ are sufficiently close in the uniform metric, then they are homotopic by Proposition~\ref{prop:flow:homotopyClassFacts}. Thus, the result follows from Lemma~\ref{l:kakutani:smoothing}.
\end{proof}

We are now in a position to define the Schwartzman asymptotic cycle and prove that it has the desired properties. The following result is due to Schwartzman \cite{Schwarzmann1957Annals}.

\begin{theorem} \label{t:schwarzmann}
Let $X$ denote a compact metric space, $T$ a continuous flow on $X$, and $\mu$ a $T$-ergodic Borel probability measure.
\begin{enumerate}
\item[{\rm(a)}] For each  $\phi\in C(X , \T)$, the limit
\begin{equation} \label{eq:t:schwarzman:rotphidef}
\rot(\phi;x) = \lim_{t\to \infty} \frac{\widetilde{\phi}_x(t)}{t}
\end{equation}
 exists for $\mu$-almost every $x \in X$ and does not depend on the choice of lift.
\item[{\rm(b)}] For each $\phi \in C(X,\T)$, there exists $\mathfrak{A}_\mu(\phi) \in \R$ with $\rot(\phi;x) = \mathfrak{A}_\mu(\phi)$ for $\mu$-a.e.\ $x$.
\item[{\rm(c)}] If $\phi$ and $\phi'$ are homotopic, then $ \mathfrak{A}_\mu(\phi) = \mathfrak{A}_\mu(\phi')$.
\end{enumerate}
 Thus, \eqref{e.gap.label.homom} descends to a well-defined homomorphism from $C^\sharp(X,\T)$ to $\R$.
\end{theorem}

\begin{definition} \label{def:flow:schwarzmannhom}
The induced map $\mathfrak{A}_\mu : C^\sharp(X,\T) \to \R$ is called the \emph{Schwartzman homomorphism}.
\end{definition}

\begin{proof}[Proof of Theorem~\ref{t:schwarzmann}]
Suppose given $\phi \in C(X,\T)$. If $\phi$ is $C^1$ along the flow, then Birkhoff's theorem (Theorem~\ref{t:birkhoff:flow}) yields
\begin{align*}
\rot(\phi;x)
& =
\lim_{t \to \infty} \frac{\widetilde \phi_x(t)}{t} \\
& =
\lim_{t \to \infty} \frac{1}{t} \left( \int_0^t \! (\partial \phi)(T^s x) \, ds + \widetilde \phi_x(0) \right) \\
& =
\lim_{t \to \infty} \frac{1}{t}  \int_0^t \! (\partial \phi)(T^s x) \, ds \\
& =
\int \partial \phi \, d\mu,
\end{align*}
for $\mu$-almost every $x \in X$. Any two lifts of $\phi_x$ differ by a fixed additive constant, which will wash out of the right hand side of \eqref{eq:t:schwarzman:rotphidef} in the limit $t \to \infty$. Thus, independence of the right hand side of \eqref{eq:t:schwarzman:rotphidef} on the choice of lift follows. Consequently, we have obtained (a) and (b) for functions that are $C^1$ along the flow.

Given a general $\phi \in C(X,\T)$, there exists a function $\psi$ that is $C^1$ along the flow and to which $\phi$ is homotopic by Corollary~\ref{coro:flowc1:homotop}. Using Proposition~\ref{p:nullhomo:lifting}, we can produce a continuous (hence bounded!) map $\widetilde f: X \to \R$ such that $\phi = \psi + \pi \circ \widetilde f$.  Boundedness of $\widetilde f$ implies that $ \rot(\phi;x) = \rot(\psi;x) $ for all $x$ for which the latter is well-defined. In particular, we obtain (a) and (b) for general $\phi$. Moreover, the preceding argument also implies that $\mathfrak{A}_\mu$ is well-defined modulo homotopy, as claimed in (c).

Now that  the various well-definedness issues above are resolved, one can verify that $\mathfrak{A}_\mu$ defines a homomorphism,  %as we have (for $\mu$-almost every $x \in X$) \begin{align*}\mathfrak{A}_\mu ([\phi] + [\psi])& =\mathfrak{A}_\mu([\phi+\psi]) \\& =\lim_{t \to \infty} \frac{\widetilde{\phi}_{x}(t) + \widetilde{\psi}_{x}(t)}{t} \\& =\lim_{t \to \infty} \frac{\widetilde{\phi}_{x}(t)}{t}+\lim_{t \to \infty} \frac{\widetilde{\psi}_{x}(t)}{t} \\& =\mathfrak{A}_{\mu} ([\phi] ) + \mathfrak{A}_{\mu} ([\psi]),\end{align*}
concluding the argument.
\end{proof}

\begin{remark}
Since $X$ is assumed to be a compact metric space, $C(X,\T)$ is separable, and so we can reverse the quantifiers: there is a uniform set $X_*$ of full $\mu$-measure so that $ \rot(\phi;x) $ exists and equals $\mathfrak{A}_\mu(\phi)$ for all $x \in X_*$ and every continuous $\phi:X \to \T$.
\end{remark}

To construct the desired gap labelling scheme, we will need one final piece: we will relate the Schwartzman homomorphism to a natural notion of rotation number of a uniformly hyperbolic continuous cocycle. To that end, suppose $\Phi$ is a uniformly hyperbolic continuous $\SL(2,\R)$-cocycle. We wish to define the rotation number of $\Phi$ to be the average winding number of $\Phi^t(x) \cdot v$ around the origin for some vector $v \in \R^2 \setminus \{0\}$. Since the Schwartzman homomorphism was defined in terms of maps into $\T$ and the winding of $\Phi^t(x)\cdot v$ is most readily comprehended as a net change in argument in $\R\PP^1$, the following definition is helpful.

\begin{definition} \label{def:flow:DeltaArg}
For $\theta \in \R$, write $e_\theta = [\cos\theta,\sin\theta]^\top$, and recall that $\bar{e}_\theta \in \R\PP^1$ denotes the corresponding equivalence class. Write $p:\R \to \R\PP^1$ for the universal covering map
\begin{equation}
p:\theta \mapsto \bar{e}_\theta \in \R\PP^1
\end{equation}
and define the identification $h:\R\PP^1 \to \T$ by
\begin{equation} \label{eq:RP1:T:isomorphism}
h: \bar{e}_{\pi\theta} \mapsto \theta \in \T.
\end{equation}
Consider a continuous map $\Lambda:[a,b]\to\R\PP^1$.
As before, one may lift $\Lambda$ to the universal cover $p:\R \to \R\PP^1$ to obtain a continuous $\widetilde{\Lambda}:[a,b] \to \R$ for which
\begin{equation}
p \circ \widetilde\Lambda = \Lambda.
\end{equation}
With this setup, one defines
\begin{equation} \label{eq:flow:DeltaArgDef}
\Delta_{\arg}^{a,b}\Lambda = \widetilde{\Lambda}(b) - \widetilde{\Lambda}(a),
\end{equation}
which we refer to as the \emph{total change of argument} of $\Lambda$ as $t$ increases from $a$ to $b$.
\end{definition}

We are now prepared to define the application of the Schwartzman homomorphism to a uniformly hyperbolic continuous cocycle $\Phi$. Given such a cocycle $\Phi$, let $\Lambda^u:X \to \R\PP^1$ denote the corresponding unstable section. Using the map $h$ from Definition~\ref{def:flow:DeltaArg}, we can view $\Lambda^u$ as a continuous map $\Lambda_0^u = h \circ \Lambda^u : X \to \T$ and define
\begin{equation} \label{eq:cocycle:schw:def}
\mathfrak{A}_\mu(\Phi)
=
\mathfrak{A}_\mu(\Lambda_0^u).
\end{equation}
Using the definitions and \eqref{eq:RP1:T:isomorphism}, we have
\begin{equation} \label{eq:cocycle:schw:arg}
\mathfrak{A}_\mu(\Lambda_0^u)
=
\lim_{t \to \infty} \frac{1}{\pi t} \Delta_{\arg}^{0,t} \left( \Lambda^u(T^s x) \right),
\quad
\mu\text{-a.e.\ } x.
\end{equation}

The content of the final theorem of this section is that \eqref{eq:cocycle:schw:def} is indeed an appropriate notion of the winding number of a hyperbolic cocycle, in that it measures the average rate of change of the argument of $\Phi^t(x) \cdot v $ for $\mu$-almost every $x$ and every $v\in \R\PP^1$.

\begin{theorem} \label{t:uh:schwarzmann}
Suppose $\Phi$ is uniformly hyperbolic. For $\mu$-almost every $x \in X$ and every $v \in \R\PP^1$, we have
\begin{equation} \label{eq:uh:schwarzmann}
\mathfrak{A}_\mu( \Phi )
=
\lim_{t \to \infty} \frac{1}{\pi t} \Delta_{\arg}^{0,t} \left( \Phi^s(x) \cdot v \right).
\end{equation}
\end{theorem}

\begin{proof}
By invariance, we have $ \Phi^s(x) \cdot \Lambda^u(x) = \Lambda^u(T^s x) $.  In particular, \eqref{eq:cocycle:schw:arg} implies that
$$
\mathfrak{A}_\mu(\Phi)
=
\mathfrak{A}_\mu(\Lambda_0^u)
=
\lim_{t \to \infty} \frac{1}{\pi t} \Delta_{\arg}^{0,t} \left( \Phi^s(x) \cdot \Lambda^u(x) \right)
$$
for $\mu$-almost every $x \in X$.  Since $\Phi^s(x)$ is nonsingular for all $s$ and $x$, it follows that
$$
\left|
\Delta_{\arg}^{0,t} \left( \Phi^s(x)  \cdot v \right)
-
\Delta_{\arg}^{0,t} \left( \Phi^s(x)  \cdot \Lambda^u(x) \right)
\right|
<
\pi
$$
for all $t>0$, $x \in X$, and every $v \in \R\PP^1$. In light of \eqref{eq:cocycle:schw:arg}, this implies that
$$
\lim_{t \to \infty} \frac{1}{\pi t} \Delta_{\arg}^{0,t} \left( \Phi^s(x) \cdot v \right)
=
\lim_{t \to \infty} \frac{1}{\pi t} \Delta_{\arg}^{0,t} \left( \Phi^s(x) \cdot \Lambda^u(x) \right)
=
\mathfrak{A}_\mu(\Phi)
$$
for any $v \in \R\PP^1$ and $\mu$-almost every $x \in X$, as desired.
\end{proof}

\begin{remark}
We want to emphasize that we work with ``merely continuous'' flows, that is, actions of $\R$ on compact metric spaces via homeomorphisms, rather than, say, actions of $\R$ on differentiable manifolds via diffeomorphisms. This is not simply generality for its own sake. Rather, the main application of the Schwartzman homomorphism will involve the suspension of a general topological dynamical system, and some of the main examples we have in mind (like a minimal translation of a totally disconnected compact group or a shift on a suitable sequence space) are given by suspensions of minimal transformations of totally disconnected spaces (which then are not differentiable manifolds).
\end{remark}

\section{The Gap Labelling Theorem} \label{sec:gaplabel}

Having discussed the necessary background, we return to the setting of ergodic Schr\"odinger operators. Recall that $(\Omega,T)$ is a topological dynamical system, $\mu$ is a $T$-ergodic Borel probability measure on $\Omega$, and $f \in C(\Omega,\R)$. By restricting $T$ if necessary, we will also assume
\begin{equation}\label{e.fullsupport}
\supp\mu = \Omega,
\end{equation}
that is, we assume $\mu(U) > 0$ for every nonempty open set $U \subseteq \Omega$. Recall also that $\Sigma_{\mu,f}$ denotes the $\mu$-almost sure spectrum of the family $\{H_{f,\omega}\}_{\omega \in \Omega}$ defined in \eqref{eq:HfomegaDef} and \eqref{eq:VfomegaDef}.
Since $\Sigma_{\mu,f}$ is closed and bounded, there exists a countable pairwise disjoint family of open intervals $G_n$ such that
\begin{equation}
 \Sigma_{\mu,f}
=
\left[\min\Sigma_{\mu,f},\max\Sigma_{\mu,f}\right] \setminus \bigcup_{n} G_n.
\end{equation}
Each such interval is naturally called a \emph{gap} of $\Sigma$. In Theorem~\ref{t.as83}, we have seen that the almost-sure spectrum $\Sigma$ is precisely the set of points of increase of the integrated density of states, $k$.  In particular, $k$ is constant on each gap. As mentioned before, the constant value $k$ assumes on $G_n$ is referred to as the \emph{label} of the gap.  In this section, we will discuss the proof of Theorem~\ref{t:gablabel}, and emphasize that it applies to \emph{any} topological family of Schr\"odinger operators satisfying the stated restrictions:  any gap label must be in the image of the Schwartzman homomorphism of the suspension of the underlying dynamical system $(\Omega,T)$ (see Definitions~\ref{def:flow:suspension} and \ref{def:flow:schwarzmannhom}).  In particular, the possible gap labels that might appear depend only on the underlying dynamics, and do not depend on the choice of sampling function $f$. Moreover, we now reap a reward from our work in Section~\ref{sec:flow}. In particular, we established the general theory surrounding the Schwartzman homomorphism without additional assumptions on the topology of $\Omega$ (such as connectedness or a smooth structure), so we can now discuss gap labelling for a general topological dynamical system. 

Recall the Schr\"odinger cocycle
\begin{equation}
(T, A_E) : \Omega \times \R^2 \to \Omega \times \R^2
\end{equation}
defined by $(T, A_E)(\omega, v) = (T\omega, A_E(\omega) v)$, where
\begin{equation}
A_E(\omega)
=
\begin{bmatrix}
E - f(T\omega) & -1 \\
1 & 0
\end{bmatrix}.
\end{equation}
As usual, we define $A_E^n$ for $n \in \Z$ by 
\begin{equation}
A_E^n(\omega) =
\begin{cases}
A_E(T^{n-1}\omega)A_E(T^{n-2}\omega) \cdots A_E(\omega) 
& n \geq 1 \\
I & n=0 \\
[A_E^{-n}(T^n\omega)]^{-1} & n \leq -1,
\end{cases}
\end{equation}
 so that iterates of the cocycle obey $ (T,A_E)^n = (T^n,A_E^n) $. The definitions of $A_E$ and $A_E^n$ ensure that one may characterize solutions of the difference equation associated to $H_\omega$; that is, $u:\Z\to\C$ solves $H_\omega u = Eu$ if and only if
 \begin{equation}
 \begin{bmatrix} u(n+1) \\ u(n) \end{bmatrix} =  A_E^n(\omega) \begin{bmatrix} u(1) \\ u(0) \end{bmatrix} \quad \text{ for every } n \in \Z.
 \end{equation}
   By \eqref{e.fullsupport} and Johnson's theorem, $\Sigma$ is precisely the complement of the energies at which $(T,A_E)$ is uniformly hyperbolic \cite{Johnson1986JDE}; see also the surveys \cite{DFLY2016DCDS, Zhang2020JST}.

Now, let $X = X(\Omega,T)$ denote the suspension of $(\Omega,T)$ (see Definition~\ref{def:flow:suspension}). In order to relate the Schwartzman homomorphism on $X$ to the integrated density of states, we want to define an interpolated $A_E^t:\Omega \to \SL(2,\R)$ for $t \in \R$ in such a way that $A_E^t$ induces a well-defined cocycle on $X$ and the winding number of orbits of $A_E^t$ corresponds to the rotation number as measured by counting sign flips of Dirichlet eigenfunctions. In particular, one may do this in such a way that $A_E^t(\omega)$ is a smooth function of $t$ for each fixed $E \in \R$ and $\omega \in \Omega$. To that end let $\theta$ and $\lambda$ be smooth nondecreasing functions so that
\begin{equation}
\theta(t) = 0 \text{ for } 0 \leq t \leq 1/6, \quad
\theta(t) = \pi/2 \text{ for } 1/3 \leq t \leq 1/2,
\end{equation}
and
\begin{equation}
\lambda(t) = 0 \text{ for } 1/2 \leq t \leq 2/3, \quad
\lambda(t) = 1 \text{ for } 5/6 \leq t \leq 1,
\end{equation}
and then define
\begin{equation} \label{def:homotopy:identity}
Y_E(\omega,t)
=
\begin{cases}
\begin{bmatrix}
\cos(\theta(t)) & -\sin(\theta(t)) \\
\sin(\theta(t)) & \cos(\theta(t))
\end{bmatrix}
&
0 \leq t \leq 1/2
\\[5.5mm]
\begin{bmatrix}
\lambda(t)(E - f(T \omega)) & -1 \\
1 & 0
\end{bmatrix}
&
1/2 \leq t \leq 1
\end{cases}
\end{equation}

With this, we define $A_E^t(\omega)$ by using $Y_E$ to interpolate between $A_E^n$ and $A_E^{n+1}$. More precisely, put
\begin{equation}
A_E^t(\omega)
=
Y_E\left(T^n \omega,t-n \right) A_E^{n}(\omega),
\quad
\omega \in \Omega, \; n \le t < n+1,
\end{equation}
where $n \in \Z$. One can check that $A_E^t(\omega)$ is a smooth function of $t$ for all fixed $\omega \in \Omega$ and $E \in \R$ that agrees with $A_E^n$ when restricted to the integers. This induces a continuous cocycle $\Phi_E$ on $X$ via
\begin{equation}
\Phi_E^t([\omega,s])
=
A_E^{t+s}(\omega)
A_E^s (\omega)^{-1}.
\end{equation}
The reader may readily check the following:

\begin{prop}\label{prop:gaplabel:phiprops} The map $\Phi_E$ is a well-defined continuous cocycle on $X$. Moreover, $\Phi_E$ is uniformly hyperbolic if and only if the corresponding discrete cocycle $(T,A_E)$ is uniformly hyperbolic.
\end{prop}

By Johnson's theorem and the relationship between $\Phi$ and $A$, Proposition~\ref{prop:gaplabel:phiprops} demonstrates that $\Phi_E$ is uniformly hyperbolic if and only if $E \notin \Sigma$. We are heading towards a proof of the identity
\begin{equation} \label{eq:gapLabelGoal}
k(E)
=
1 - \mathfrak{A}_{\overline{\mu}}(\Phi_E),
\end{equation}
where $\mathfrak{A}_{\overline{\mu}}$ denotes the Schwartzman homomorphism (defined in Definition~\ref{def:flow:schwarzmannhom}). Let us recall some notation from Subsection~\ref{subsec:schwartzman}: $C^\sharp(X,\T)$ denotes the set of homotopy classes of maps $X \to \T$, and $\mathfrak{A}_{\overline{\mu}}:C^\sharp(X,\T) \to \R$ denotes the Schwartzman homomorphism, which is defined by \eqref{e.gap.label.homom}. Finally, $\mathfrak{A}_{\overline{\mu}}(\Phi)$ for a hyperbolic cocycle $\Phi$ is defined to be the Schwartzman homomorphism evaluated at the (homotopy class of the) unstable section of said cocycle. 

\begin{remark} \label{rem:gl:AcontainsZ}Notice that the image of the Schwartzman homomorphism in the present case always contains $\Z$ (proof: use the map $X \ni [\omega,t] \mapsto nt \in \T$). In particular, \eqref{eq:gapLabelGoal} suffices to establish that $k(E)$ lies in the image of $\mathfrak{A}_{\overline{\mu}}$ whenever $E$ lies in a gap.
\end{remark}

From Theorem~\ref{t:ids:rotnumber}, we know that $k$ can be related to sign flips of solutions, namely
\begin{equation} \label{eq:ids:defn}
1 - k(E)
=
\lim_{N \to \infty} \frac{1}{N} \# \left\{ 1 \leq j \leq N : \sgn(u_{E,\omega}(j)) \neq \sgn(u_{E,\omega}(j+1)) \right\}
\end{equation}
for $\mu$-almost every $\omega \in \Omega$, where $u_{E,\omega}$ is the Dirichlet solution of $H_\omega u = Eu$ (with the convention $\sgn(0) = 1$). Our main goal is to relate this sign-flip counting to a rotation number in a more explicit fashion in order to invoke the Schwartzman homomorphism. If $\Lambda:[a,b] \to \R\PP^1$ is continuous, recall that $\Delta_{\arg}^{a,b}\Lambda(t)$ denotes the total change in the argument of $\Lambda(t)$ as $t$ increases from $a$ to $b$ as in \eqref{eq:flow:DeltaArgDef}.

\begin{theorem} \label{t:ids:rotnumber:continuous}
For each $E \in \R $ and $\mu$-almost every $\omega$,
\begin{equation} \label{eq:ids:rotation}
1 - k(E)
=
\lim_{N \to \infty} \frac{1}{\pi N} \Delta_{\arg}^{0,N}(A_E^t(\omega)  \cdot e_1).
\end{equation}

\end{theorem}

\begin{proof}
Fix $E \in \R$ and $\omega$ in the full-measure set for which \eqref{eq:ids:defn} holds and suppress $\omega$ from the notation. For each $t$, define
\begin{equation}
\textbf{u}(t)
=
\begin{bmatrix}
u(t+1) \\ u(t)
\end{bmatrix}
=
A^t_E \cdot e_1.
\end{equation}
Note that invertibility of $A_E$ implies $\mathbf{u}(t) \neq 0$ for all $t \in \R$. By our choice of $\omega$, and our definition of $u(t)$, \eqref{eq:ids:defn} implies that
\begin{equation}
1-k(E)
=
\lim_{N \to \infty} \frac{1}{N} \#
\{ j \in \N : 1 \leq j \leq N \text{ and } \sgn(u(j)) \neq \sgn(u(j+1))\}
\end{equation}
for all $E$. The key observation here is that each sign flip of $u(j)$ corresponds to a half-rotation of $\textbf{u}(t)$ about the origin and that the cocycle $A^t_E$ is such that one actually sees these rotations reflected in $\Delta_{\arg}(\textbf{u}(t))$; the rest is bookkeeping. Concretely, we claim that
\[
\#
\{ j \in \N : 1 \leq j \leq N \text{ and }\sgn(u(j)) \neq \sgn(u(j+1))\}
=
\frac{1}{\pi} \Delta_{\arg}^{0,N+1} (\textbf{u}(t)) + O(1).
\]
Let us make this more precise. Given signs $s,s' \in \{+, \, - \}$, define the quadrant $Q^{s,s'} = \{(x,y) \in \R^2 : sx > 0, s'y > 0 \}$. Similarly, denote the four semi-axes by
\begin{equation}
 Q^{0,\pm} = \{(0,y) : \pm y > 0\},
\quad
 Q^{\pm,0} = \{ (x,0) : \pm x > 0 \}.
\end{equation}
One can show the following:
\begin{equation} \label{eq:ids:rotnumber}
\#
\{ j \in \N \cap[1,N] : \sgn(u(j)) \neq \sgn(u(j+1))\}
=
\begin{cases}
\left\lfloor \frac{1}{\pi} \Delta_{\arg}^{0,N+1} (\textbf{u}(t)) \right\rfloor - 1 & \mathbf u_N \in Q^{0,+} \\[3mm]
\left\lfloor \frac{1}{\pi} \Delta_{\arg}^{0,N+1} (\textbf{u}(t)) \right\rfloor & \text{otherwise}
\end{cases}
\end{equation}
for all $N \in \N$, which suffices to prove \eqref{eq:ids:rotation}. To prove \eqref{eq:ids:rotnumber}, one proceeds inductively as in the proof of Theorem~\ref{t:sturm:osc}. Denote the left-hand side of \eqref{eq:ids:rotnumber} by $F_N$ and the right-hand side by $R_N$.

First, consider the case $N=1$.  If $u(2) > 0$, then $\sgn(u(1)) = \sgn(u(2))$ and $u(2) \neq 0$, so $u(t) > 0$ for all $0 < t \leq 2$. In particular, $\textbf{u}(t)$ is in the open upper half-plane for all $0 < t \leq 2$, and hence
\[
0 \leq \Delta_{\arg}^{0,2}(\textbf{u}(t)) < \pi.
\]
If $u(2) = 0$ (notice that this implies $\mathbf{u}(1) \in Q^{0,+}$), then one has $ \Delta_{\arg}^{0,2} \mathbf{u}(t) =\pi$ by a direct calculation. On the other hand, if $\sgn(u(1)) \neq \sgn(u(2))$, one has
$$
\pi
\leq
\Delta_{\arg}^{0,2}(\textbf{u}(t))
<
2 \pi.
$$
Thus, the claim holds for $N=1$ (keep in mind the convention $\sgn\, 0 = 1$). Now, suppose that $F_{N-1} = R_{N-1}$ for some $N \ge 2$. One must consider several cases separately.

\begin{case} \textbf{{\boldmath$\mathbf u(N) \in Q^{-,0}$.}} First, notice that this implies that $\sgn(u(N+1)) \neq \sgn(u(N))$, which means that the sign flip counter increments, i.e.,
\[
F_N = F_{N-1} + 1.
\]
Additionally, in this case, one has $\mathbf u(N-1) \in Q^{0,+}$ and $\mathbf u(N+1)$ is in the open lower half-plane. In particular,
\[
R_N
=
\left\lfloor \frac{1}{\pi} \Delta_{\arg}^{0,N+1} (\textbf{u}(t)) \right\rfloor \\
=
\left\lfloor \frac{1}{\pi} \Delta_{\arg}^{0,N} (\textbf{u}(t)) \right\rfloor
=
R_{N-1}+1,
\]
which proves $R_N = F_N$ in this case. Notice that we have used $\mathbf u(N-1) \in Q^{0,+}$ to get the final equality.
\end{case}

\begin{case} \textbf{{\boldmath$\textbf{u}(N) \in Q^{-,+}$}} Since this implies $\sgn(u(N)) \neq \sgn(u(N+1))$, one has $F_N = F_{N-1} + 1$. Notice that $\textbf{u}(N+1) \in Q^{-,-} \cup Q^{0,-} \cup Q^{+,-}$.  Using the explicit form of the homotopy used to construct $A_E^t$, we see that
\[
R_N
=
\left\lfloor \frac{1}{\pi} \Delta_{\arg}^{0,N+1}(\textbf{u}(t)) \right\rfloor
=
\left\lfloor \frac{1}{\pi} \Delta_{\arg}^{0,N}(\textbf{u}(t)) \right\rfloor + 1
=
R_{N-1}+1,
\]
which shows $F_N = R_N$ in this case.
\end{case}

The remaining cases are similar and left to the reader. Thus, $F_N = R_N$ for every $N \in \N$, so \eqref{eq:ids:rotnumber} holds.
\end{proof}

\begin{theorem}\label{t:gap:label} 
With setup as above, let $G \subseteq \R \setminus \Sigma_{\mu,f}$ be given. We have
\begin{equation} \label{eq:ids:schwarzmann}
1 - k(E)
=
\mathfrak{A}_{\overline{\mu}}(\Phi_E)
\end{equation}
for every $E \in G$.
\end{theorem}

\begin{proof}
By Theorem~\ref{t:ids:rotnumber:continuous},
\begin{equation} \label{eq:ids:winding}
1 - k(E)
=
\lim_{N \to \infty} \frac{1}{\pi N} \Delta_{\arg}^{0,N} (A_E^t(\omega) e_1),
\quad
\mu \text{-a.e.\ } \omega \in \Omega.
\end{equation}
On the other hand, Theorem~\ref{t:uh:schwarzmann} implies
\begin{equation}
\mathfrak{A}_{\overline{\mu}}(\Phi_E)
=
\lim_{N \to \infty} \frac{1}{\pi N} \Delta_{\arg}^{0,N} (\Phi_E^t(x) e_1),
\quad
\overline{\mu} \text{-a.e.\ } x \in X.
\end{equation}
Consequently,
\begin{equation} \label{eq:uhenergy:schwarzmann}
\mathfrak{A}_{\overline{\mu}}(\Phi_E)
=
\lim_{N \to \infty} \frac{1}{\pi N} \Delta_{\arg}^{0,N} (A_E^t(\omega) e_1),
\quad
\mu \text{-a.e.\ } \omega \in \Omega.
\end{equation}
Combining \eqref{eq:ids:winding} and \eqref{eq:uhenergy:schwarzmann}, we get \eqref{eq:ids:schwarzmann}.
\end{proof}

This allows us to complete the proof of Theorem~\ref{t:gablabel}.

\begin{proof}[Proof of Theorem~\ref{t:gablabel}]
As discussed above, we know that $\schwartzmanGroup = \schwartzmanGroup(\Omega,T,\mu)$ contains $\Z$. Since $\schwartzmanGroup$ is a subgroup of $\R$, if $E$ belongs to a gap, we get
\[k(E) = 1-(1-k(E)) \in \schwartzmanGroup\]
 from Theorem~\ref{t:gap:label}. Since $0 \le k(E) \le 1$ for all $E$ by definition, the theorem is proved.
\end{proof}

\begin{remark}
Recall that we showed in Proposition~\ref{prop:flow:homotopyClassFacts} that $C^\sharp(\Omega,\T)$ is countable whenever $\Omega$ is compact. Since (by Theorem~\ref{t:schwarzmann}), the Schwartzman homomorphism is constant on homotopy classes, it follows that the range of the Schwartzman homomorphism is a \emph{countable} set. Thus, there is a fixed countable set of gap labels that only depends on the base dynamics $(\Omega,T)$.
\end{remark}

\section{Almost-Periodic Potentials}\label{sec:ap}

With the general theorem proved, it is naturally of interest to compute the group $\schwartzmanGroup(\Omega,T,\mu)$ for specific examples. In this section, we begin this endeavor by computing $\schwartzmanGroup(\Omega,T,\mu)$ for almost-periodic potentials. For reference, this section will address the claims from Example~\ref{ex:ap}.

We will start off by studying two special cases of almost-periodic potentials: periodic and quasi-periodic potentials. In each of those two cases, one can identify the groups $\schwartzmanGroup(\Omega,T)$ in a few lines, whereas the characterization of $\schwartzmanGroup(\Omega,T)$ for general almost-periodic hulls takes more work.

\subsection{Examples}

As a warm-up, let us see how the gap labels arise from the Schwartzman homomorphism in the setting of periodic Schr\"odinger operators. Of course, one can produce the gap labels directly from Floquet theory; the point is to demonstrate the general theory in the simplest example first.

Suppose $V:\Z \to \R$ satisfies $V(n+p) = V(n)$ for some $p \in \N$ and every $n \in \Z$. This can be realized as an ergodic operator with 
\begin{equation}
\Omega = \Z_p, \quad T[k] = [k + 1], \quad f([k]) = V(k),
\end{equation}
 where $[k]$ denotes the class of $k \in \Z$ in $\Z_p$. Of course, $(\Z_p,T)$ is minimal and uniquely ergodic with unique invariant measure given by normalized counting measure.

\begin{prop} \label{prop:ap:schwartzOfZp}
Let $p \in \N$, $\Omega = \Z_p$, and $T:\Omega \to \Omega$ be given by $T[k]=[k+1]$. We have
\[\schwartzmanGroup(\Omega,T) = \frac{1}{p}\Z = \set{\frac{n}{p} : n \in \Z}.\]
\end{prop}

\begin{proof}
The reader can verify that $X=X(\Omega,T)$ is homeomorphic to the circle $\T$ via the identification  given by
\begin{equation} \label{eq:suspOfZp}
X \ni\big[[k],t\big] \mapsto p^{-1}(k+t) \in \T.
\end{equation}
Since every continuous map $\T \to \T$ is homotopic to $x \mapsto nx$ for some $n \in \Z$, \eqref{eq:suspOfZp} implies that every map $\phi:X \to \T$ is homotopic to
\[\phi^{(n)}: [0,t] \mapsto nt/p \in \T \]
for some $n \in \Z$ (notice that this defines $\phi^{(n)}$ for all $x\in X$, since $[\bar k, t] = [0,t+k]$). For any choice of base point $x \in X$, one considers $\phi_x^{(n)}:\R \to \T$ given by $\phi_x^{(n)}(t) = \phi^{(n)}(T^tx)$ as in Section~\ref{subsec:schwartzman}. One readily sees that the lift $\widetilde{\phi}^{(n)}_x$ is of the form $nt/p+C$ for some $C=C(x)$ and thus
\[\lim_{t\to\infty} \frac{\widetilde\phi_x^{(n)}(t)}{t}
= \lim_{t\to \infty} \frac{nt/p+ C}{t} = \frac{n}{p},\]
 which shows that the range of $\mathfrak{A}$ is contained in $p^{-1}\Z$, as claimed. Moreover, since every $\phi^{(n)}$ is continuous, the previous calculation also shows that the range  contains $p^{-1}\Z$.
\end{proof}

Before considering the case of general almost-periodic potentials, let us consider the class of quasi-periodic potentials, which as we will see in Theorem~\ref{t:ap:freqmod:char} is a strict superset of the set of periodic potentials and a strict subset of the set of all almost-periodic potentials. Here, the base dynamics are given by $R_\alpha: \omega\mapsto\omega+\alpha$ with $\alpha \in \R^d$ having rationally independent coordinates.\footnote{Recall that we write $\alpha$ both for the vector in $\R^d$ and its projection to $\T^d$.} In this case, it is known that $(\T^d,R_\alpha)$ is strictly ergodic with Lebesgue measure supplying the unique invariant measure; see \cite[Propositions~1.4.1, 4.2.2, and 4.2.3]{KatokHassel1995Book}.

\begin{theorem} \label{t:ap:qpschwartzmangroup}
Let $\alpha \in \R^d$ have rationally independent coordinates. One has
\[\schwartzmanGroup(\T^d,R_\alpha) = \Z + \alpha\Z^d = \set{k_0 + \sum_{j=1}^d k_j\alpha_j : k_j \in \Z \ \forall 0 \le j \le d}.\]
\end{theorem}

\begin{proof}
The suspension $X = X(\T^d,R_\alpha)$ is homeomorphic to $\T^{d+1}$ via the map
\begin{equation}
X \ni [\omega,t] \mapsto [\omega+t\alpha,t].
\end{equation}
Since every continuous map $\T^{d+1} \to \T$ is homotopic to\footnote{This is well known and not hard to show using that any map $\T \to \T$ is homotopic to $x\mapsto nx$ for some $n \in \Z$. Alternatively, it also follows from the more general discussion below in Section~\ref{subsec:apgen}.} $x\mapsto \langle k,x \rangle$ for some $k \in \Z^{d+1}$, the result follows from a calculation similar to the conclusion of the proof of Proposition~\ref{prop:ap:schwartzOfZp}.
\end{proof}

\subsection{Generalities about Almost-Periodic Sequences}
We now move to a discussion of general almost-periodic sequences. The first goal is to place a general almost-periodic potential in the setting of dynamically defined potentials. One can accomplish this by showing that the shift on the hull of an almost-periodic sequence can be viewed as a translation of a suitable compact group.

First, we recall some background without proofs. Let $S$ denote the shift $[SV](n) = V(n+1)$ for $V \in \ell^\infty(\Z)$ and $n \in \Z$. Given $V \in \ell^\infty(\Z)$, the \emph{orbit} of $V$ is the set
\[\orb(V) = \set{V(\cdot - k) : k \in \Z}\]
of all translates of $V$. The \emph{hull} of $V$ is the closure of the orbit in $\ell^\infty$:
\[ \hull(V) = \overline{\orb(V)}^{\ell^\infty(\Z)}. \]
We say that $V$ is \emph{almost-periodic} if $\hull(V)$ is compact in $\ell^\infty(\Z)$. The reader can readily check that a given $V \in \ell^\infty(\Z)$ is almost-periodic if and only if for all $\varepsilon>0$, the $\varepsilon$ almost-periods of $V$ are relatively dense in $\Z$ (recall $p\in\Z$ is an $\varepsilon$ almost-period of $V$ if $\|V-S^pV\|_\infty < \varepsilon$).

Given $V \in \ell^\infty(\Z)$ almost-periodic, it is well known (and not hard to show) that the pairing $\oast:\orb(V) \times \orb(V) \to \orb(V)$ given by
\begin{equation}
S^kV \oast S^\ell V := S^{k+\ell}V
\end{equation}
is uniformly continuous, and thus extends uniquely to a uniformly continuous binary operation on $\hull(V)$. One has:

\begin{prop}
If $V$ is almost-periodic, then $(\hull(V),\oast)$ is a compact abelian topological group with a dense cyclic subgroup {\rm(}namely $\orb(V)${\rm)}.
\end{prop}

The allows one to characterize almost-periodic sequences as precisely those that are generated by continuously sampling translations of compact monothetic groups:

\begin{prop} \label{p:ap:char:mintrans}
Let $V \in \ell^\infty(\Z)$ be given.
The following are equivalent:
\begin{enumerate}
\item[{\rm(a)}] $V$ is almost-periodic.
\item[{\rm(b)}] There exist a compact monothetic metrizable group $\Omega$ {\rm(}with dense cyclic subgroup generated by $\alpha${\rm)}, a continuous function $f:\Omega \to \R$, and an element $\omega_0 \in \Omega$ such that
 \begin{equation}\label{e.apdyndefrep}
V(n) = f(\omega_0 + n\alpha), \quad n \in \Z.
\end{equation}
\end{enumerate}
\end{prop}

\begin{proof}
$\boxed{(\mathrm{a})\implies(\mathrm{b}).}$ If $V$ is almost-periodic, we can write it in the form \eqref{e.apdyndefrep} with $\Omega = \hull(V)$, $\omega_0 = V$, $f(\omega) = \omega(0)$ (which is continuous), and $\alpha = SV$. The reader should recall that we write the group operation as $\oast$ in this case.

$\boxed{(\mathrm{b})\implies(\mathrm{a}).}$ Assume that $V$ is of the form \eqref{e.apdyndefrep}. For each $\omega \in \Omega$, define $V_\omega \in \ell^\infty(\Z)$ by
\begin{equation}
V_\omega(n)
=
f(\omega + n\alpha),
\quad
n \in \Z.
\end{equation}
%First, notice that $V_\omega$ is indeed a member of $\hull(V)$ for all $\omega \in \Omega$.  To see this, use density of the subgroup generated by $\alpha$ to pick $n_j \in \Z$ with $\omega_0 + n_j\alpha \to \omega$ as $j \to \infty$. Thus, $$\lim_{j \to \infty} \| S^{n_j}V - V_\omega\|_\infty=0$$
%by uniform continuity of $f$, so indeed $V_\omega \in \hull(V)$.  Next, continuity of the map $\omega \mapsto V_\omega$ is a consequence of uniform continuity of $f$.  Finally, to show that the map is surjective, let $V' \in \hull(V)$ be given.  Choose $m_k \in \Z$ so that $S^{m_k} V \to V'$ as $k \to \infty$. Using compactness of $\Omega$ and passing to a subsequence if necessary, we may assume that $\omega_0 +m_k\alpha$ converges to an element $ \omega' \in \Omega $ as $k \to \infty$.  Again, uniform continuity of $f$ implies that $S^{m_k} V \to V_{\omega'}$, that is, $V' = V_{\omega'}$. Thus, $\hull(V)$ is the continuous image of a compact topological space and so is itself compact.
The reader may readily check that $V_\omega \in \hull(V)$ for every $\omega \in \Omega$ and that the mapping $\omega \mapsto V_\omega$ is continuous and maps $\Omega$ onto $\hull(V)$. Thus, the hull of $V$ is the continuous image of a compact set and hence is compact.
\end{proof}

In this section we single out several prominent subclasses of almost-periodic sequences, prove that they are indeed almost-periodic, and characterize them in terms of their hulls.
Let us begin with the genuinely periodic sequences. The following is not hard to check and is left to the reader.

\begin{prop} \label{p:per:hull:char}\mbox{}
\begin{enumerate}
\item[{\rm(a)}] Every periodic sequence is almost-periodic.
\item[{\rm(b)}] An almost-periodic sequence is periodic if and only if its hull is isomorphic to $\Z_p$ for some $p \in \N$.
\item[{\rm(c)}] An almost-periodic sequence is aperiodic if and only if no point of its hull is isolated.
\end{enumerate}
\end{prop}

%\begin{proof}
%If $V$ is periodic, then $\orb(V) = \{V,SV,\ldots, S^{p-1}V\}$, where $p$ is the minimal period of $V$. In particular, $\hull(V) = \orb(V)$ is finite, hence compact. One can see that the map which sends $S^k V$ to the residue class of $k$ modulo $p$ defines an isomorphism from $\hull(V)$ to $\Z_p$.

%To see the nontrivial implication in the final claim, suppose $\hull(V)$ contains an isolated point. Since the hull is a topological group, it follows that all of its points are isolated, so $\hull(V)$ is finite, by compactness. But then $\hull(V) = \orb(V)$ is a finite cyclic group, as desired.\end{proof}

\begin{definition}
We say that $V \in \ell^\infty(\Z)$ is \emph{limit-periodic}\index{limit-periodic} if it belongs to the $\ell^\infty$-closure of the set of periodic points of $S$. More precisely, $V$ is limit-periodic if there is a sequence $(V_j)_{j=1}^\infty$ such that $V_j \in \ell^\infty(\Z)$ for each $j$, $V_j$ is periodic for every $j$, and
\begin{equation}
\lim_{j \to \infty} \|V - V_j\|_\infty
=
0.
\end{equation}
\end{definition}

\begin{prop} \label{p.almost.per.lim.per}\mbox{}
\begin{enumerate}
\item[{\rm(a)}] Every limit-periodic sequence is almost-periodic.
\item[{\rm(b)}] An almost-periodic sequence is limit-periodic if and only if its hull is totally disconnected.
\end{enumerate}
\end{prop}

\begin{proof}
(a) Suppose $V$ is limit-periodic. For each $j \in \N$, choose $V_j \in \ell^\infty(\Z)$ and $p_j \in \N$ such that $S^{p_j}V_j = V_j$ and $\lim_{j \to \infty} \|V - V_j\|_\infty = 0$. It suffices to prove that $\orb(V)$ is totally bounded in order to show that $V$ is almost-periodic. Given $\varepsilon > 0$, choose $j$ such that
$$
\|V - V_j\|_\infty < \varepsilon.
$$
We claim that $\orb(V)$ is contained in the following finite union of $\varepsilon$-balls:
$$
\bigcup_{k = 0}^{p_j - 1} B(S^kV_j,\varepsilon).
$$
Indeed, given $\ell \in \Z$, choose $0 \le k < p_j$ with $\ell \equiv k \mod p_j$. Then, by $p_j$-periodicity of $V_j$ and isometry of $S$, we have
$$
\|S^\ell V - S^k V_j\|_\infty
=
\|S^\ell V - S^\ell V_j\|_\infty = \|V - V_j\|_\infty < \varepsilon.
$$
Thus, $\hull(V)$ is compact and $V$ is almost-periodic. 

(b) First, supposing that $V$ is limit-periodic, let us show that $\hull(V)$ is totally disconnected. Since the hull of $V$ is a topological group, it suffices to show that there are arbitrarily small neighborhoods of the identity that are both open and closed. Therefore, given $\varepsilon > 0$ we will show that $B(V,\varepsilon) \cap \hull(V)$ contains a non-empty set that is both closed and open. Since $V$ is limit-periodic, choose a periodic $W$ with period $p$ so that $\|W - V\|_\infty \leq \frac{\varepsilon}{2}$. Then,
\begin{equation}
\hull^p(V) = \overline{\{ S^{jp} V : j \in \Z \}}
\end{equation}
is a compact subgroup of $\hull(V)$ of index at most $p$. By construction, $\hull^p(V)$ is closed, but it is also open since it is the complement of the union of no more than $p-1$ other closed cosets. Moreover, since $S^p W = W$ and $\|W - V\|_\infty \leq \frac{\varepsilon}{2}$, every element $V' \in \hull^p(V)$ satisfies $ \| V' - W \|_\infty \leq \varepsilon/2 $ and hence $\hull^p(V)$ is contained in the $\varepsilon$-ball around $V$. Consequently, $\hull(V)$ is totally disconnected.

Conversely, let $V$ be almost-periodic and suppose its hull is totally disconnected. We have to show that $V$ is limit-periodic. Given $\varepsilon > 0$, we have to find $W \in \ell^\infty(\Z)$ and $p \in \N$ with $S^p W = W$ and $\|W - V\|_\infty < \varepsilon$. By total disconnectedness, we may choose a compact open neighborhood $N$ of $V$ in $\hull(V)$, small enough so that
\begin{equation}\label{e.avilalp1}
\|(W_1 \oast W_2) - W_1 \|_\infty
<
\varepsilon/2
\quad
\text{ for every } W_1 \in \hull(V) \text{ and every } W_2 \in N.
\end{equation}
This is possible since $\oast$ is uniformly continuous, $V$ is the identity of $\hull(V)$ with respect to $\oast$, and $\hull(V)$ is totally disconnected. Since the sets $N$ and $\hull(V) \setminus N$ are compact and disjoint, there exists $\delta > 0$ so that $\| X - Y \| \geq \delta$ for all $X \in N$ and all $Y \in \hull(V) \setminus N$. By almost-periodicity of $V$, we can choose $p \ge 1$ so that
$$
\|S^p V - V\|_\infty < \delta,
$$
hence $S^p V \in N$ by our choice of $\delta$. But then we find inductively that $\{ S^{kp}V : k \in \Z \} \subseteq N$, by isometry of $S$ and the choice of $\delta$. Now consider the $p$-periodic $W$ that coincides with $V$ on $[0,p-1]$. Given $n \in \Z$, we write $n = r + \ell p$ with $\ell \in \Z$ and $0 \le r \le p-1$. Then, it follows from \eqref{e.avilalp1} that
$$
|V(n) - W(n)|
=
| V(r + \ell p) - V(r)|
=
\left| (S^r V \oast S^{\ell p}V) (0) - (S^rV)(0) \right|
<
\varepsilon/2,
$$
since $S^{\ell p}V \in N$. This shows that the $p$-periodic $W$ obeys $\|W - V\|_\infty \leq \varepsilon/2 < \varepsilon$, concluding the proof.
\end{proof}

\begin{definition}
We say that $V \in \ell^{\infty}(\Z)$ is \emph{quasi-periodic}\index{quasi-periodic} if and only if it can be obtained by continuously sampling along orbits of a translation on a finite-dimensional torus.  More precisely, $V$ is quasi-periodic if and only if it can be written as
\begin{equation}\label{eq:qp:def}
V(n)=f(n\alpha + \omega_0),
\quad n \in \Z,
\end{equation}
for some $\omega_0,\alpha \in \T^d$, $d \in \N$, and some continuous $f: \T^d \to \R$. 
\end{definition}

Naturally, if $d=1$ and $\alpha = 1/p \in \T$, then $f(\omega_0+n \alpha)$ defines a $p$-periodic sequence. Moreover, it is clear that any given $p$-periodic sequence can be obtained in this manner.

\begin{prop} \label{p:qp:char}
\mbox{}
\begin{enumerate}
\item[{\rm(a)}] Any quasi-periodic sequence is almost-periodic. 
\item[{\rm(b)}] An almost-periodic sequence is quasi-periodic if and only if its hull is isomorphic to $\T^r \oplus A$ for some $r \geq 0$ and some finite abelian group $A$.
\end{enumerate}
\end{prop}

\begin{proof}(a) This follows immediately from the definition and Proposition~\ref{p:ap:char:mintrans}.

(b) To prove one direction, let us assume that $V$ is quasi-periodic and show that the hull of $V$ is indeed a direct sum of a finite-dimensional torus and a finite abelian group. Since $V$ is quasi-periodic, we may write $V(n) = f(n\alpha+\omega_0)$ with $\omega_0,\alpha \in \T^d$ and $f \in C(\T^d)$. Replacing $f$ by a suitable translate thereof, we may assume $\omega_0=0$. Define $K(\alpha)$ to be the closed subgroup of $\T^d$ generated by $\alpha$, and put
\begin{equation}
V_\omega(n) = f(n\alpha + \omega),
\quad
n \in \Z, \, \omega \in K(\alpha),
\end{equation}
as in the proof of Proposition~\ref{p:ap:char:mintrans}. Arguing as before, we see that $\Phi(\omega) = V_\omega$ defines a continuous surjective map $\Phi:K(\alpha) \to \hull(V)$. Moreover $\Phi$ is a homomorphism from $K(\alpha)$ to $\hull(V)$, and hence is a (topological) quotient map by the open mapping theorem (see, e.g., \cite[Chapter~1]{Morris1977:TopGrps}). In particular, $\hull(V)$ is isomorphic to $ K(\alpha) / \ker(\Phi) $, which must be of the stated form since it is a quotient of a closed subgroup of $\T^d$ by a closed subgroup.

Conversely, assume that $V$ is almost-periodic and that $\hull(V) \cong \T^r \oplus A$.  By using the classification of finite abelian groups (for example), it is not hard to see that $\hull(V)$ is isomorphic to a closed subgroup $K \subseteq \T^d$ for some $d \geq r$. Let $\widetilde f:K \to \hull(V)$ be a continuous group isomorphism, and choose $\alpha$ so that $\widetilde f(\alpha) = SV$. Extending $\widetilde{f}$ to a continuous map $\widetilde f:\T^d \to \hull(V)$, we see that $V$ can be realized via \eqref{eq:qp:def} with $\omega_0 = 0$ and $ f(\omega) = \left[\widetilde f(\omega)\right](0)$.
\end{proof}

\subsection{The Frequency Module} \label{subsec:apgen}

As discussed in Propsition~\ref{p:ap:char:mintrans}, every almost-periodic potential is a dynamically defined potential with base dynamics given by a minimal translation of a compact abelian group; moreover, the structure of the group dictates the structure of the potential.

\begin{definition}
Given a (locally) compact abelian topological group $G$, a \emph{character} of $G$ is a continuous homomorphism $\chi: G \to \T$. The collection $\widehat{G}$ of all characters of $G$ is a group under pointwise addition that is itself a topological group in the compact-open topology, called the \emph{dual group} of $G$.
\end{definition}

We freely use well known facts about topological groups, their duals, and harmonic analysis on such groups. For reference and additional background, see any of the texbook treatments in \cite{DiestelSpalsbury2014, Higgins1974, Katznelson2004:HA3rdEd, Loomis1953:HA, MontgomeryZippin1955, Morris1977:TopGrps, Rudin1990Fourier}.

\begin{definition}
Let $V$ be an almost-periodic sequence, $\Omega_0 = \orb(V)$, $\Omega = \hull(V)$, and $\Psi: \Z \to \Omega$ the natural homomorphism
\begin{equation}
\Psi: k
\mapsto
S^k V.
\end{equation}
Since $\widehat{\Z} \cong \T$ (by identifying $\chi \in \widehat{\Z}$ with $\chi(1) \in \T$), this in turn induces a homomorphism $\widehat{\Psi}: \widehat{\Omega} \to \T $ by duality. Thus
\begin{equation} \label{def:fm:homo}
\widehat{\Psi}(\chi)
=
\chi(SV),
\quad
\chi \in \widehat{\Omega}.
\end{equation}
Since $\Omega_0$ is a dense subgroup of $\Omega$, $\widehat{\Psi}$ is an injective homomorphism from $\widehat{\Omega}$ onto its image in $\T$, so we may identify $\widehat\Omega$ with a countable subgroup of $\T$. The pullback of this subgroup of $\T$ to $\R$ is called the \emph{frequency module} of $V$, and is denoted by $\M = \M(V)$.  In particular, since $\Z$ is the kernel of the canonical projection from $\R$ to $\T$, one has $\Z \subseteq \M$.
\end{definition}

Using the explicit form of $\Psi$ in \eqref{def:fm:homo}, we see that $t \in \M(V)$ if and only if there exists $\chi \in \widehat \Omega$ with $\chi(SV) = t\, (\mathrm{mod} \, \Z)$, which holds if and only if
\begin{equation}
\chi_t(S^k V)
=
kt \, (\mathrm{mod} \, \Z)
\end{equation}
extends to a continuous homomorphism $\chi_t:\Omega \to \T$.

\begin{theorem} \label{t:ap:gaplabel}
Suppose $\Omega$ is a compact topological group and $\alpha \in \Omega$ generates a dense cylic subgroup. With $\widehat{\varphi}_\alpha$, $R_\alpha$, and $\pi$ defined as in Example~\ref{ex:ap}, one has
\begin{equation}
\schwartzmanGroup(\Omega,R_\alpha) = \pi^{-1}(\widehat{\varphi}_\alpha(\widehat{\Omega})).
\end{equation}
In particular, if $V$ is almost-periodic, $\Omega = \hull(V)$, and $T:\Omega \to \Omega$ denotes the shift $[T\omega](n) = \omega(n+1)$, then
\begin{equation}
\schwartzmanGroup(\Omega,T) = \M(V).
\end{equation}
\end{theorem}

Before proving the main theorem, it is instructive to try to understand these definitions when the underlying potential is periodic, quasi-periodic, or limit-periodic. We will show how one can characterize these various subclasses in terms of their frequency modules. To do this adequately, we need one more definition. Let us say that $G$ has the \emph{divisor property} if, for every pair  $g , g' \in G$, there exist $n,n'\in \Z$ and an element $d \in G$  such that $nd = g$ and $n'd = g'$.

\begin{lemma} \label{t:freqmod:vs:omegadual}
Let $V$ be almost-periodic, $\Omega = \hull(V)$, and $\M = \M(V)$.  
\begin{enumerate}
\item[{\rm(a)}] $\M$ is finitely generated  if and only if $\widehat \Omega$ is finitely generated. 
\item[{\rm(b)}] $\M$ enjoys the divisor property if and only if $\widehat{\Omega}$ is a torsion group.
\end{enumerate}
\end{lemma}

\begin{proof}
The proof of (a) is left to the reader. For the second, first suppose that every element of $\widehat \Omega$ has finite order, and note that this implies $\M \subseteq \Q$.  Thus, given $r = p/q$ and $r' = p'/q'$ in $\M$, use Bezout's theorem to choose $n,m \in \Z$ with $d = npq' + mp'q$, where $d$ denotes the greatest common divisor of $p'q$ and $pq'$. Naturally,
\begin{equation}
nr + mr'
=
\frac{d}{qq'} \in \M
\end{equation}
is a divisor of both $r$ and $r'$.

Conversely, suppose $\M$ has the divisor property. Since $1 \in \M$, one has $\M \subseteq \Q$, which in turn implies that every element of $\widehat \Omega$ has finite order, concluding the proof of part~(b).
\end{proof}

\begin{theorem} \label{t:ap:freqmod:char}
Let $V$ be almost-periodic and $\M = \M(V)$.  Then
\begin{enumerate}
\item[{\rm(a)}] $V$ is periodic if and only if $\M$ is discrete.
\item[{\rm(b)}] $V$ is quasi-periodic if and only if $\M$ is finitely generated.
\item[{\rm(c)}] $V$ is limit-periodic if and only if $\M$ enjoys the divisor property.
\end{enumerate}
\end{theorem}

\begin{proof} Put $\Omega = \hull(V)$.

(a) If $V$ is periodic, one can calculate directly that $\M$ is the subgroup of $\R$ generated by $1/p$, where $p$ is the minimal period of $V$.  On the other hand, if $\M$ is discrete, then $\widehat\Omega$ must be finite. By Pontryagin duality, $\Omega$ is finite, so $V$ is periodic.

(b) By Lemma~\ref{t:freqmod:vs:omegadual}, it suffices to prove that $V$ is quasi-periodic if and only if $\widehat\Omega$ is finitely generated.  By the classification of finitely generated abelian groups and Pontryagin duality, $\widehat\Omega$ is finitely generated if and only if $\Omega$ is of the form $\T^r \oplus A$ for some finite abelian group $A$ and some $r \geq 0$. By Proposition~\ref{p:qp:char}, $\hull(V)$ is of this form if and only if $V$ is quasi-periodic.

(c) By Proposition~\ref{p.almost.per.lim.per}, $V$ is limit-periodic if and only if $\Omega$ is totally disconnected. The reader can check that $\Omega$ is totally disconnected if and only if $\widehat{\Omega}$ is a torsion group, so this statement follows from the second claim in Lemma~\ref{t:freqmod:vs:omegadual}.
\end{proof}

\begin{coro}
A sequence is periodic if and only if it is both quasi-periodic and limit-periodic.
\end{coro}

\begin{proof}
A periodic sequence is trivially both quasi-periodic and limit-periodic. Conversely, if $V$ is quasi-periodic and limit-periodic, then $\M$ is finitely generated and has the divisor property. As this implies that $\M$ is generated by a single element, it follows that $V$ is periodic by Theorem~\ref{t:ap:freqmod:char}.
\end{proof}

We conclude by recovering the classical gap labelling theorem for almost-periodic potentials. The main observations are that the suspension $X = X(\Omega,R_\alpha)$ is itself a topological group and that continuous functions on $X$ are homotopic to characters of $X$.

\begin{proof}[Proof of Theorem~\ref{t:ap:gaplabel}] Let $\Omega$ be given and suppose $\alpha \in \Omega$ generates a dense cyclic subgroup. We have already discussed $\Omega = \Z_p$, so let us assume that $\Omega$ is infinite. Denote by $\mu$ the normalized Haar measure on $\Omega$, and consider the suspension $X = X(\Omega,T)$. 

In $X$, notice that
\begin{equation}
[\omega,s] = [\omega',s'] \iff (\omega-\omega',s-s') = (m\alpha,-m) \text{ for some } m \in \Z.\end{equation}
Thus, $X$ is the quotient of the topological group $\Omega \times \R$ by the closed subgroup $\{(m\alpha,-m):m \in \Z\}$, so $X$ is a topological group, and moreover the subgroup $S = \{ [0,t] : t \in \R \}$ is dense in $X$. The topology of $X$ is metrizable, with a metric given by
\[
d_X(x_1,x_2)
=
\inf\set{d_{\Omega}(\omega_1,\omega_2) + | t_1 - t_2 | : \omega_j \in \Omega, \; t_j \in \R \text{ with } [\omega_j,t_j] = x_j}.
\]
We will also denote by $d$ the usual metric on $\T$:
\begin{equation}
d_\T(t,t')
= \dist(t - t',\Z).
\end{equation}
Now, let $\phi \in C(X,\T)$ be given, and denote $\rho = \mathfrak{A}_{\overline \mu}(\phi)$. Using uniform continuity of $\phi$, one can check that $\rot(\phi;x)=\rho$ for every $x \in X$, in particular for\footnote{We are writing the abelian group $\Omega$ additively, so one should understand the zero in the first coordinate as the additive identity element of $\Omega$.} $x=[0,0]$.

%We first claim that $\rot(\phi;x) = \rho$ for every $x \in X$. We know that this holds for $\overline\mu$-almost every $x$ and hence for a dense set of $x \in X$, since $\overline\mu$ is precisely normalized Haar measure on $X$. Notice that, by uniform continuity, we may choose $\varepsilon > 0$ with the property that $d(\phi(x), \phi(x') ) < 1/2$ whenever $x,x' \in X$ satisfy $d(x,x') < \varepsilon$. Then, since the translation flow is isometric, we get $\rot(\phi;x) = \rot(\phi;x')$ whenever  $d(x,x') < \varepsilon$. Since $\Omega$ is infinite monothetic, $\Omega$ has no isolated points, so it follows that $\rot(\phi;x)$ exists and equals $\rho$ for every $x$, in particular for $x=[\alpha,0]$.

Recall that $\pi:\R \to \T$ denotes the canonical projection, choose a continuous function $\widetilde\phi:\R \to \R$ so that $\phi([0,t]) =\pi(\widetilde\phi(t))$, and denote 
\[\widetilde G(t) = \widetilde\phi(t) - \rho t, \quad G= \pi \circ \widetilde G.
\] Notice that these definitions yield
\begin{equation} \label{eq:tildeGiso(t)}
\lim_{t \to \infty} \frac{\widetilde G(t)}{t}
= \lim_{t\to\infty} \frac{\widetilde{\phi}(t)-\rho t}{t}
=\rho-\rho
=
0.
\end{equation}
To conclude, we must show $\pi(\rho) \in \widehat{\varphi}_\alpha(\widehat{\Omega})$. We will need the following claim.

\begin{claim} The map $S\to\T$ given by $[0,t]\mapsto \pi(\rho t)$ is uniformly continuous.
\end{claim}

\begin{claimproof} Since $\phi$ is continuous, it suffices to prove that $[0,t] \mapsto G(t)$ is uniformly continuous\footnote{Notice that this is stronger than just proving uniform continuity of $\widetilde G$, since $d([\alpha,t],[\alpha,s])$  can be small even if $s-t$ is not small.} on $S$.  To that end, let $0 < \varepsilon < 1/2$ be given. Since $\phi$ is continuous on the compact space $X$, it is uniformly continuous, and we may find $\delta > 0$ so that
\[
d_\T(\phi(x), \phi(x')) < \varepsilon
\text{ whenever }
d_X(x,x') < \delta.
\]
Consequently, a typical lifting argument shows that
\begin{equation}
\left|\widetilde\phi(t) - \widetilde\phi(s)\right|
<
\varepsilon
\text{ whenever }
|t-s| < \delta.
\end{equation}
Thus, $\widetilde\phi$ is uniformly continuous, and so too is $\widetilde G$. Now, suppose that $d_X([0,0],[0,q]) < \delta$ for some $q \in \N$ and $0 < \delta < 1$ (note that this forces $d_\Omega(q\alpha,0)<\delta$). There then exists an integer $n$ such that
\[
\left|\widetilde\phi(t+q) - \widetilde\phi(t) - n\right|
<
\varepsilon
\]
for all $t \in \R$. Equivalently,
\begin{equation}
|\widetilde G(t + q) - \widetilde G(t) -(n- q\rho)|
<
\varepsilon
\end{equation}
for all $t$. Next, let us show that $|n- q\rho| > \varepsilon$, produces a contradiction to \eqref{eq:tildeGiso(t)}. Concretely, if $n - q\rho > \varepsilon$, then
\[
\widetilde G(t+q) - \widetilde G(t)
\geq
\Delta
:=
n - q \rho - \varepsilon > 0
\]
for all $t$. However, this necessarily implies $\liminf \widetilde G(t)/t \geq \Delta/q > 0$, contradicting \eqref{eq:tildeGiso(t)}; the case $n - q\rho < -\varepsilon$ is similar. Consequently, one has $|n- q\rho| \leq \varepsilon$ and so
\[
|\widetilde G(t+q) - \widetilde G(t)|
<
2\varepsilon
\]
for all $t \in \R$. Putting everything together,
\[
d_\T(G(t),G(s)) < 3\varepsilon
\text{ whenever }
d_X([0,t], [0,s]) < \delta,
\]
and hence $S \ni [0,t] \mapsto G(t)$ is uniformly continuous, which in turn implies that $[0,t] \mapsto \pi(\rho t)$ is uniformly continuous on $S$.
\end{claimproof}
\bigskip

By the claim, the map $\psi:S \to \T$ given by $\psi:[0,t]\mapsto \pi(\rho t)$ is uniformly continuous, and hence extends to a continuous function $\psi:X \to \T$. Restricting $\psi$ to the fiber 
\[ \{ [\omega,0] : \omega \in \Omega \} \cong \Omega,
\]
 we obtain an element of $\widehat \Omega$ via $\chi(\omega) = \psi([\omega,0])$. One has 
 \[\widehat{\varphi}_\alpha(\chi) = \chi(\alpha) = \psi([\alpha,0]) = \psi([0,1]) = \pi(\rho),\] which shows $\pi(\rho) \in \widehat{\varphi}_\alpha(\widehat{\Omega})$ and completes the proof of $\schwartzmanGroup(\Omega,R_\alpha) \subseteq \pi^{-1}(\widehat{\varphi}_\alpha(\widehat{\Omega}))$.

 To see the reverse inclusion, let $\rho \in \pi^{-1}(\widehat{\varphi}_\alpha(\widehat{\Omega}))$ be given. Then, there exists $\chi \in \widehat{\Omega}$ such that $\chi(\alpha) = \pi(\rho)$.  Notice that $\chi$ can be extended to an element $\chi \in\widehat{X}$ via
\begin{equation}
\chi([\omega,t]) = \pi(\rho t) + \chi(\omega),
\quad
\omega \in \Omega, \; t \in \R.
\end{equation}
One can check that $\mathfrak{A}_{\overline{\mu}}(\chi) = \rho$, and hence $\pi^{-1}(\widehat{\varphi}_\alpha(\widehat{\Omega})) \subseteq \schwartzmanGroup(\Omega,R_\alpha)$.
\end{proof}

\section{Subshift Potentials} \label{sec:subshift}

Our next class of examples will be those topological dynamical systems given by subshifts. The present section addresses the claims from Example~\ref{ex.subshift}.

\subsection{Schwartzman Group Associated with a General Subshift}
Suppose $\mathcal{A}$ is a finite set (called the \emph{alphabet}) and give $\mathcal{A}^\Z$ the product topology induced by equipping each factor with the discrete topology. A set $\Omega \subseteq \mathcal{A}^\Z$ is called a \emph{subshift} over $\A$ if it is closed (thus compact) and invariant under the action of the shift $T:\mathcal{A}^\Z \to \mathcal{A}^\Z$ given by $[T\omega](n)=\omega(n+1)$. Abusing notation somewhat, we also write $T$ for the restriction of $T$ to $\Omega$ since this should not cause confusion.

The topology on $\A^\Z$ (hence on $\Omega$) is metrizable, e.g., by
\[ d(\omega,\omega') = 2^{-\min\{|n| : \omega(n) \neq \omega'(n) \} }, \quad \omega \neq \omega'. \]
One may naturally ask how to construct continuous functions from $X = X(\Omega,T)$ to $\T$ whenever $(\Omega,T)$ is a subshift, and moreover, how to construct a family that is large enough that it represents every homotopy class of maps $X \to \T$. The idea is to start with a continuous integer-valued function, $g$, on the subshift and use those to construct functions on the suspension by winding around the circle $g(\omega)$ times in passing from $[\omega,0]$ to $[\omega,1]=[T\omega,0]$.

More precisely, given $g \in C(\Omega,\Z)$, define $\og:X \to \T$ by
\begin{equation} \label{eq:overfDef}
\og[\omega,t] = t g(\omega),
\quad \omega \in \Omega, \ t \in [0,1].
\end{equation}
Given a subshift $\Omega \subseteq \A^\Z$, recall that a \emph{cylinder set} is a set of the form
\[ \Xi = \set{\omega \in \Omega : \omega_{n+j} = u_j \text{ for all } 1 \le j \le k}, \]
where $u= u_1\cdots u_k \in \A^*$ and $n \in \Z$.

\begin{theorem}\label{t.subshiftproperties}
Let $(\Omega,T)$ denote a subshift and let $X = X(\Omega,T)$ denote its suspension. 
\begin{enumerate}
\item[{\rm(a)}] For any $g \in C(\Omega,\Z)$, $\og$ defined in \eqref{eq:overfDef} belongs to $C(X,\T)$.\smallskip

\item[{\rm(b)}] $C^\sharp(X,\T) = \{[\og] : g \in C(\Omega,\Z)\}$. That is, every $h \in C(X,\T)$ is homotopic to $\og$ for some $g \in C(\Omega,\Z)$.
\smallskip

\item[{\rm(c)}] If $\mu$ is a $T$-ergodic measure on $\Omega$, then 
\begin{equation}
\schwartzmanGroup(\Omega,T,\mu) = \set{\int \! g \, d\mu : g \in C(\Omega,\Z)}
\end{equation}
\smallskip

\item[{\rm(d)}] If $\mu$ is a $T$-ergodic measure on $\Omega$, then $\schwartzmanGroup$ is the $\Z$-module generated by 
\[S = \set{\mu(\Xi) : \Xi \subseteq \Omega \text{ is a cylinder set}}.
\]
\end{enumerate}
\end{theorem}

\begin{proof} (a) Let $g \in C(\Omega,\Z)$ be given. Let us first note that
\[\og[\omega,1]=0=\og[T\omega,0],\]
so $\og$ is well-defined. Since $g$ is continuous and integer-valued, choose $\delta>0$ such that $d(\omega,\omega')<\delta$ implies $g(\omega) = g(\omega')$. Suppose one is given $\omega,\omega'$ such that  $d(\omega,\omega')<\delta$, and denote $k=g(\omega)= g(\omega')$. One then can note
\[ d_\T(g[\omega,t], g[\omega',t']) \leq k |t-t'|, \quad t,t'\in[0,1]\]
to see that $\og$ is continuous.
\medskip

\noindent (b) Let $h:X \to \T$ be a given continuous function. 

\begin{step} Using total disconnectedness of $\Omega$, one can check that there exists $h_1 \in C(X,\T)$ homotopic to $h$ such that $h_1[\omega,0]=0$ for all $\omega \in \Omega$. %To see this, choose $\varepsilon>0$ small and pick $\delta>0$ small enough that \[d\big(g([\omega,t]), g([\omega',t']) \big)< \varepsilon \text{ whenever } d(\omega,\omega')<\delta \text{ and } d(t,t')<\delta. \]  Write $\Omega$ as  disjoint union $\Omega = \Omega_1 \sqcup \cdots \sqcup \Omega_k$, where each $\Omega_j$ is closed and has diameter less than $\delta$; note then that each $\Omega_j$ is also open. Perform a homotopy on $\Omega_j \times [-\eta,\eta]$ for some $0<\eta<\delta$ by fixing the values of $g([\omega,\pm\eta])$, moving $g([\omega,0])$ to $0$, and extend to the identity on the remainder of the suspension.
\end{step}

\begin{step} The winding of $t \mapsto h_1[\omega,t]$ is locally constant. More precisely, for each $\omega \in \Omega$, lift $\phi_\omega:t \mapsto h_1[\omega,t]$ to  $\widetilde\phi_\omega:\R \to \R$ and define 
 \[g(\omega) = \widetilde\phi_\omega(1) - \widetilde\phi_\omega(0).\]
  One has $g \in C(\Omega,\Z)$. Indeed, if $\omega$ and $\omega'$ are sufficiently close, then $d(\phi_\omega(t),\phi_{\omega'}(t)) < 1/2$ for all $t$, so one can choose lifts that preserve this property and hence $g(\omega) = g(\omega')$. Thus, $g \in C(\Omega,\Z)$, as claimed.
 \end{step}
 
 \begin{step} One can check that $h_1$ is homotopic to $\bar{g}$, which concludes the proof of (b).
 \end{step}
 
(c) Given $g \in C(\Omega,\Z)$ and $\omega \in \Omega$, consider $\phi_\omega(t) = \og[\omega,t]$ and the corresponding lift $\widetilde{\phi}_\omega$. From the definition of $\og$, one has, for $\mu$-a.e.\ $\omega$,
\[ \lim_{n\to\infty} \frac{\widetilde{\phi}_\omega(n)}{n} = \lim_{n\to\infty} \frac{1}{n}\sum_{j=0}^{n-1} g(T^j\omega) = \int \! g\, d\mu. \]
Thus, the conclusion of (c) immediately by noting that the previous calculation implies
\[\mathfrak{A}_{\overline{\mu}}(\og) = \int \! g \, d\mu.\]
 (d) Note that any integer combination of characteristic functions of cylinder sets lies in $C(\Omega,\Z)$, so the $\Z$-module generated by $S$ is contained in $\schwartzmanGroup(\Omega,T,\mu)$. For the other inclusion, consider $g:\Omega \to \Z$ continuous. By continuity and compactness, $f$ assumes finitely many values $s_1,\ldots,s_k$ and the sets $\Omega_j =g^{-1}(\{s_j\})$, $1 \le j \le k$ are closed and open. Writing each $\Omega_j$ as a disjoint union of cylinder sets, one sees that $\int\! g \, d\mu$ is in the $\Z$-module generated by $S$.
\end{proof}

\begin{coro}\label{c.SgroupUESH}
Let $(\Omega,T)$ denote a strictly ergodic subshift. We have
\begin{equation}\label{e.SgroupUESH}
\schwartzmanGroup(\Omega,T) = \set{\int g \, d\mu : g \in C(\Omega,\Z)},
\end{equation}
where $\mu$ denotes the unique $T$-invariant measure on $\Omega$.
\end{coro}

\subsection{Subshifts Generated by Substitutions} \label{subsec:substitution}
The following result contains the statement from part~(2) of Example~\ref{ex.subshift}.

\begin{theorem}\label{t.fibonacci}
Let $(\Omega,T)$ denote the Fibonacci subshift. If $\alpha = \frac{1}{2}(\sqrt{5}-1)$ denotes the inverse of the golden mean, then
\begin{equation}
\schwartzmanGroup(\Omega,T) = \Z[\alpha].
\end{equation}
\end{theorem}

We will derive Theorem~\ref{t.fibonacci} from a more general result for primitive substitution subshifts, since the latter is obtained via a procedure that works in the same fashion for any subshift of that kind. This procedure is due to Bellissard, Bovier, and Ghez \cite{BelBovGhe1992}.

\begin{definition}
 A \emph{substitution} on the finite alphabet $\A$ is given by a map $S:\A \to \A^*$, where $\A^*$ denotes the set of finite words over $\A$. The substitution $S$ can be extended by concatenation to $\A^*$ and $\A^\N$. A fixed point $u$ of $S : \A^\N \to \A^\N$ is called a \emph{substitution sequence}; the associated subshift is given by
$$
\Omega = \Omega_u = \{ \omega \in \A^\Z : \text{each finite subword of $\omega$ is a subword of u} \}.
$$
\end{definition}

In the Fibonacci case considered in Example~\ref{ex.subshift} we have $\A = \{ 0,1 \}$, $S(0) = 01$, $S(1) = 0$, and $u = 0100101001001 \ldots$.

\begin{definition}
The \emph{substitution matrix} $M = M_S$ associated with a substitution $S : \A \to \A^*$, where $\A = \{ a_1 , \ldots, a_\ell \}$, is given by the $\ell \times \ell$ matrix $M = (m_{i,j})_{1 \le i,j \le \ell}$, where $m_{i,j}$ is given by the number of times $a_i$ occurs in $S(a_j)$. The substitution $S$ is called \emph{primitive} if its substitution matrix $M$ is primitive, that is, for some $k \in \N$, all entries of $M^k$ are strictly positive.
\end{definition}

The Fibonacci substitution is primitive, which follows from observing that
\begin{equation}\label{e.fibsubstmat}
M^2 = \begin{bmatrix} 2 & 1 \\ 1 & 1 \end{bmatrix}.
\end{equation}

If $S$ is primitive and $u$ is a substitution sequence, then the associated subshift $\Omega$ is known to be strictly ergodic \cite{Queffelec1987}. In fact, even if $S : \A^\N \to \A^N$ does not have a fixed point, one can check that some power of this map does, and that power will be primitive as well and hence the associated subshift will be strictly ergodic. In addition, the resulting subshift is independent of the power and the fixed point one chooses and therefore only depends on $S$. 

We are thus in a setting covered by Corollary~\ref{c.SgroupUESH} and wish to determine $\schwartzmanGroup(\Omega,T)$ via \eqref{e.SgroupUESH}. In order to do so, we need the following concepts. Applying the Perron-Frobenius theorem to $M$ we infer that $M$ has a leading simple eigenvalue $\theta > 0$ (in the sense that it has multiplicity one and every other eigenvalue $\lambda$ of $M$ obeys $|\lambda| < \theta$), along with an eigenvector $v_\theta$ with strictly positive entries. We will normalize $v_\theta$ so that the sum of its entries equals one. The entries of this normalized vectors then correspond to the frequencies with which the symbols $a_j$ appear in $u$ (or any $\omega \in \Omega$) \cite{Queffelec1987}. 

In order to study frequencies of subwords of $u$ of length $m > 1$, one can consider the set $\CW_m$ of these words and the derived substitution $S_m : \CW_m \to \CW_m^*$, which is defined as follows. If $w = w_1 \ldots w_m \in \CW_m$ and $S(w) = s_1 s_2 \ldots s_{|S(w)|}$ with $w_j, s_j \in \A$, then set
\begin{equation}\label{e.smdef}
S_m(w) = (s_1 \ldots s_m) (s_2 \ldots s_{m+1}) \ldots (s_{|S(w_1)|} \ldots s_{|S(w_1)| + m - 1}).
\end{equation}
It can be checked that $S_m$ is primitive as well, with the same leading eigenvalue $\theta$ and an $m$-dependent normalized corresponding eigenvector $v_\theta^{(m)}$ and, as before, the frequency of a subword $w \in \CW_m$ is given by the entry of $v_\theta^{(m)}$ that corresponds to $w$ \cite{Queffelec1987}.

We can now state the general result for subshifts arising from primitive substitutions, which is due to Bellissard, Bovier, and Ghez \cite{BelBovGhe1992}.

\begin{theorem}\label{t.primitivesubst}
If $S$ is a primitive substitution, then with the notation from above, the Schwartzman group $\schwartzmanGroup(\Omega,T)$ is given by the $\Z[\theta^{-1}]$--module generated by the entries of $v_\theta$ and $v_\theta^{(2)}$.
\end{theorem}

Let us first show how Theorem~\ref{t.fibonacci} follows quickly from Theorem~\ref{t.primitivesubst}.

\begin{proof}[Proof of Theorem~\ref{t.fibonacci}]
As pointed out above, in the case of the Fibonacci substitution, the substitution matrix is given by \eqref{e.fibsubstmat}. A quick calculation then shows that the leading eigenvalue is given by the golden mean
$$
\theta = \frac{1 + \sqrt{5}}{2} = \alpha^{-1}
$$
and the corresponding normalized eigenvector is given by
$$
v_\theta = \begin{bmatrix} \theta^{-1} \\ \theta^{-2} \end{bmatrix} = \begin{bmatrix} \theta - 1 \\ 2 - \theta \end{bmatrix}.
$$
Another calculation shows that the substitution matrix associated with $S_2$ is given by
$$
M_2 = \begin{bmatrix} 0 & 0 & 1 \\ 1 & 1 & 0 \\ 1 & 1 & 0 \end{bmatrix},
$$
which has the same leading eigenvalue $\theta$ and associated normalized eigenvector
$$
v_\theta^{(2)} = \begin{bmatrix} \theta^{-3} \\ \theta^{-2} \\ \theta^{-2} \end{bmatrix} = \begin{bmatrix} 2\theta - 3 \\ 2 - \theta \\ 2 - \theta \end{bmatrix}.
$$
Using one more time that $\theta^{-1} = \theta - 1 = \alpha$, we see that the $\Z[\theta^{-1}]$--module generated by the entries of $v_\theta$ and $v_\theta^{(2)}$ is equal to $\Z[\alpha]$. Thus, Theorem~\ref{t.fibonacci} indeed follows from Theorem~\ref{t.primitivesubst}.
\end{proof}

\begin{proof}[Proof of Theorem~\ref{t.primitivesubst}]
Recall that due to \eqref{e.SgroupUESH} we need to consider the collection of numbers $\int\! f \, d\mu$, where $\mu$ denotes the unique $T$-invariant measure on $\Omega$ and $f$ runs through $C(\Omega,\Z)$.

Due to the nature of the topology of $\Omega$ and the fact that each $f$ in question takes values in the integers, we can restrict our attention to the consideration of the $\mu$-measure of cylinder sets; compare Theorem~\ref{t.subshiftproperties}.(d).

This in turn shows that we need to understand the frequencies with which finite subwords of $u$ occur in $u$ (or any $\omega \in \Omega$ due to unique ergodicity). As pointed out above, the frequency of a subword $w \in \CW_m$ is given by the entry of $v_\theta^{(m)}$ that corresponds to $w$. Thus, it remains to relate the entries of $v_\theta^{(m)}$ for $m \ge 3$ to the entries of $v_\theta$ and $v_\theta^{(2)}$.

A useful relation of the desired nature is obtained by observing from \eqref{e.smdef} that if $p$ is large enough, then $S_m^p(w)$ only depends on the first two symbols of $w$, that is, on $w_1, w_2$ if $w = w_1 \ldots w_m$ with $w_j \in \A$. Indeed, writing
$$
S^p(w) = S^p(w_1) \ldots S^p(w_m) = s_1 \ldots s_{|S^p(w)|}
$$
and using that $(S_m)^p = (S^p)_m$, we find
$$
S_m^p(w) =  (s_1 \ldots s_m) (s_2 \ldots s_{m+1}) \ldots (s_{|S^p(w_1)|} \ldots s_{|S^p(w_1)| + m - 1}).
$$
Thus, if
$$
|S^p(w_1)| + |S^p(w_2)| \ge |S^p(w_1)| + m - 1,
$$
then $S^p(w)$ depends only on $w_1$ and $w_2$. By primitivity of $S$ (and hence $M$), this inequality holds for large enough $p$, uniformly in $w_2$. 

Fix such a value of $p$. As a consequence of this choice, in the matrix $M_m^p$ (where $M_m$ denotes the substitution matrix associated with the substitution $S_m$), all columns with labels $w$ having the same prefix $w_1 w_2$ are equal. Thus, we may as well collapse them into one column. This results in the matrix $M_{p,m,2}$, where the columns are now labeled by the elements $w_1 w_2$ of $\CW_2$. Denoting the process of passing to the prefix of length $2$ by $\pi_{2,m} : \CW_m \to \CW_2$, $w = w_1\ldots w_m \mapsto w_1 w_2$ and the associated substitution matrix by $P_{2,m}$, we therefore have
\begin{equation}\label{e.substrelation1}
M_m^p = M_{p,m,2} P_{2,m}.
\end{equation}

On the other hand, a moment's thought yields the additional relations
\begin{equation}\label{e.substrelation2}
M_2^p = P_{2,m} M_{p,m,2}
\end{equation}
and
\begin{equation}\label{e.substrelation3}
M_m M_{p,m,2} =  M_{p+1,m,2} = M_{p,m,2} M_2.
\end{equation}

We are now able to address our actual goal: expressing the entries of $v_\theta^{(m)}$ in terms of those of $v_\theta$ and $v_\theta^{(2)}$. Consider the vector
\begin{equation}\label{e.tildevmdef}
\tilde v^{(m)} := M_{p,m,2} v_\theta^{(2)}.
\end{equation}
By \eqref{e.substrelation3} we have
$$
M_m \tilde v^{(m)} = M_m M_{p,m,2} v_\theta^{(2)} = M_{p,m,2} M_2  v_\theta^{(2)} = \theta M_{p,m,2} v_\theta^{(2)} = \theta \tilde v^{(m)}.
$$
As $\theta$ is the leading eigenvalue of $M_m$ and it is simple, it follows that $\tilde v^{(m)}$ must be a multiple of $v_\theta^{(m)}$. To determine the multiplier, let us denote by $\mathbf{1}_k$ the vector of all ones in $\R^{\# \CW_k}$, so that
\begin{align*}
\sum_{w \in \CW_m} \tilde v^{(m)}_w & = \langle \mathbf{1}_m, \tilde v^{(m)} \rangle \\
& = \langle \mathbf{1}_2 P_{2,m}, M_{p,m,2} v_\theta^{(2)} \rangle \\
& = \langle \mathbf{1}_2, P_{2,m} M_{p,m,2} v_\theta^{(2)} \rangle \\
& = \langle \mathbf{1}_2, M_2^p v_\theta^{(2)} \rangle \\
& = \langle \mathbf{1}_2, \theta^p v_\theta^{(2)} \rangle \\
& =  \theta^p,
\end{align*}
where we used \eqref{e.substrelation2} in the fourth step and the normalization of $v_\theta^{(2)}$ in the sixth step. As $v_\theta^{(m)}$ is normalized as well, we see that the multiplier is $\theta^p$, that is,  \begin{equation}\label{e.tildevmrel}
\tilde v^{(m)} = \theta^p  v_\theta^{(m)}.
\end{equation}

It now follows from \eqref{e.tildevmdef} and \eqref{e.tildevmrel} that for every $m \ge 3$, each of the entries of $v_\theta^{(m)}$ is given by an integral linear combination of the entries of $v_\theta^{(2)}$ times a negative power of $\theta$, concluding the proof.
\end{proof}

Two other prominent examples of substitutions on a two-symbol alphabet are given by the Thue--Morse and period-doubling substitutions. The \emph{Thue--Morse substitution} acts on $\A = \{0,1\}$ via $0\mapsto 01$ and $1 \mapsto 10$.
\begin{theorem} Let $(\Omega,T)$ denote the Thue--Morse subshift. One has
\begin{equation} \label{eq:subshift:thuemorsegaplabels} \schwartzmanGroup(\Omega,T) = \set{ \frac{k}{3\cdot 2^n} : k \in \Z, \ n \in \Z_+}. \end{equation} \end{theorem}
\begin{proof} The substitution matrix, leading eigenvalue, and normalized eigenvector for the Thue--Morse substitution are given by
\begin{equation} M = \begin{bmatrix} 1 & 1 \\ 1 & 1 \end{bmatrix}, \quad \theta = 2, \quad  v_\theta = \begin{bmatrix} 1/2 \\ 1/2 \end{bmatrix}. \end{equation}
Another quick calculation shows that the substitution matrix associated with $S_2$ and its leading eigenvector are given by
\begin{equation} M_2= 
\begin{bmatrix} 0 & 0 & 1 & 0 \\ 1 & 1 & 0 & 1 \\ 1 & 0 & 1 & 1 \\ 0 & 1 & 0 & 0 \end{bmatrix}, \quad v_\theta^{(2)} =  \begin{bmatrix}1/6 \\ 1/3 \\ 1/3 \\ 1/6 \end{bmatrix}. \end{equation}  
In view of Theorem~\ref{t.primitivesubst}, $\schwartzmanGroup(\Omega,T)$ is the $\Z[1/2]$-module generated by $\{1/2,1/3,1/6\}$, which the reader can readily check is equivalent to the set in \eqref{eq:subshift:thuemorsegaplabels}.\end{proof}
The \emph{period-doubling substitution} acts on $\A = \{0,1\}$ via $0\mapsto 01$ and $1 \mapsto 00$.
\begin{theorem} Let $(\Omega,T)$ denote the period-doubling subshift. One has
\begin{equation} \schwartzmanGroup(\Omega,T) = \set{ \frac{k}{3\cdot 2^n} : k \in \Z, \ n \in \Z_+}. \end{equation} \end{theorem}
\begin{proof} The substitution matrix, leading eigenvalue, and normalized eigenvector of the period-doubling substitution are given by
\begin{equation} M = \begin{bmatrix} 1 & 2 \\ 1 & 0 \end{bmatrix}, \quad \theta = 2, \quad v_\theta = \begin{bmatrix} 2/3 \\ 1/3 \end{bmatrix}. \end{equation} 
Another quick calculation shows that the substitution matrix and leading eigenvector associated with $S_2$ are given by
\begin{equation} M_2= 
\begin{bmatrix} 0 & 0 & 2 \\ 1 & 1 & 0 \\ 1 & 1 & 0 \end{bmatrix}, \quad v_\theta^{(2)} = \begin{bmatrix} 1/3 \\ 1/3 \\ 1/3 \end{bmatrix}. \end{equation} 
The result follows again from Theorem~\ref{t.primitivesubst}. \end{proof}

\subsection{Full Shift Over a Finite Alphabet}

Let us consider the random case in which $\Omega = \A^\Z$ for some finite set $\A$, $\mu = \mu_0^\Z$ for a probability measure $\mu_0$ on $\mathcal{A}$, and $T$ again is the shift. Without loss, take $\mathcal{A} = \{1,2,\ldots,m\}$, and abbreviate $w_j = \mu(\{j\})$ for $1 \le j \le m$. 

\begin{theorem} \label{t:subshift:bernoulli}
For the random setup here,
\[ \schwartzmanGroup(\Omega,T,\mu) = \Z[w_1,w_2,\ldots,w_m]. \]
In particular, in the case of equal weighting $w_1 = \cdots = w_m = 1/m$, one has
\[ \schwartzmanGroup(\Omega,T,\mu) = \Z[1/m]. \]\end{theorem}

\begin{proof}
Consider $u = a_1 \ldots a_n$ with $a_j \in \A$ for each $j$. The function
\[\chi_u(\omega) = \begin{cases} 1 & \omega|_{[1,n]} = u
\\
0 & \text{otherwise}
\end{cases}\]
is a locally constant (hence continuous) integer-valued function. Thus, 
\[ \int \chi_u\, d\mu = w_{a_1} w_{a_2} \cdots w_{a_n} \in \schwartzmanGroup(\Omega,T,\mu). \]
Since $\schwartzmanGroup$ is a subgroup of $\R$, the inclusion ``$\supseteq$'' is proved.
\medskip

Conversely, assume given $f \in C(\Omega,\Z)$. By compactness, we can write $f$ as a finite sum
\[f = \sum k \chi_{_{\Omega_k}}, \quad \Omega_k = f^{-1}(\{k\}).\]
By writing each $\Omega_k$ as a disjoint union of cylinder sets, we observe the inclusion ``$\subseteq$''.
\end{proof}

\begin{comment}
\subsection*{General Random Potentials}

\jdf{
Suppose $\A$ is a compact topological space, $\Omega = \A^\Z$, $[T\omega](n) = \omega(n+1)$, and $\mu$ is a $T$-ergodic measure. Then 
\[\schwartzmanGroup(\Omega,T,\mu) = \set{\int\! f \, d\mu : f \in C(\Omega,\Z)}.\]
\begin{proof} This is proved in the same way as Theorem~\ref{t.subshiftproperties}.\marginpar{check and recombine?}\end{proof}
}
\begin{theorem}
Suppose $\mu_0$ is a Borel probability measure with compact support $\A \subseteq \R$, and consider $\Omega = \A^\Z$, $\mu = \mu_0^\Z$, and $T:\Omega \to \Omega$ the left shift. Then $\schwartzmanGroup(\Omega,T,\mu)$ is the $\Z$-algebra generated by the set
\[ \mu(C(\A,\Z)) := \set{\int_\A f \, d\mu_0 : f \in C(\A,\Z)}. \]
\end{theorem}

\jdf{\begin{proof} Given $f_1,\ldots,f_s \in C(\A,\Z)$, define
\[f(\omega) = \prod_{j=1}^s f_j(\omega_j),\] and observe that \[ \int \! f \, d\mu = \prod_{j=1}^s \int\! f_j \, d\mu_0. \]
Since $\schwartzmanGroup(\Omega,T,\mu)$ is a subgroup of $\R$, this proves that the algebra generated by $\mu(C(\A,\Z))$ is contained in $\schwartzmanGroup(\Omega,T,\mu) = \mu(C(\Omega,\Z))$. \newline Conversely, assume given $f \in C(\Omega,\Z)$. TK UNDERSTAND THIS STEP  \end{proof}}

Notice that this recovers the result of Theorem~\ref{t:subshift:bernoulli}: if $\A$ is finite (hence discrete), then $C(\A,\Z) = \Z^\A$.
\end{comment}

\section{Potentials Generated by Affine Torus Homeomorphisms} \label{sec:affine}

In this section, we will consider homeomorphisms of the torus induced by suitable affine maps on $\R^d$ and address the claims from Example~\ref{ex:affine}. Namely, given $d \in \N$, denote $\Omega = \T^d$. Given $A \in \SL(d,\Z)$, and $b \in \T^d$, we recall $T = T_{A,b}:\Omega \to \Omega$ is the homeomorphism $\T^d \to \T^d$ given by $T_{A,b} \, \omega =   A\omega+b$. This class includes several important examples as special cases, among which we single out two (both in dimension $d=2$). The \emph{cat map} is the homeomorphism  $T_{\rm cat}$ of $\T^2$ given by
\[A= \begin{bmatrix} 2 & 1 \\ 1 & 1\end{bmatrix}, \quad b = 0. \]
Given $\alpha \in \T\setminus \Q$, the associated \emph{skew-shift} is defined by
\begin{equation}
T(\omega_1,\omega_2) = (\omega_1+\alpha,\omega_1 + \omega_2)
\end{equation}
which corresponds to the choices
\begin{equation}
A = \begin{bmatrix} 1 & 0 \\ 1 & 1 \end{bmatrix}, \quad b = \begin{bmatrix} \alpha \\ 0 \end{bmatrix}.
\end{equation}

For the affine transformations discussed in this section, the image of the Schwartzman homomorphism may be computed as follows:

\begin{theorem} \label{t:affine:Tabschwartzgrp}
Let $d \in \N$, $A \in \SL(d,\Z)$, and $b \in \T^d$ be given, and let $\mu$ be a $T_{A,b}$-ergodic measure. In this case,
\begin{equation}
\schwartzmanGroup(\T^d,T_{A,b},\mu) = \set{\langle k, b \rangle + n : n \in \Z \text{ and } k \in  \Z^d \cap \ker(I-A^*)}.
\end{equation}
\end{theorem}

Here, let us point out that Theorem~\ref{t:affine:Tabschwartzgrp} computes $\schwartzmanGroup(\T^d,T_{A,b},\mu)$ for any ergodic measure, $\mu$. However, in order for Theorem~\ref{t:gablabel} to guarantee that $\schwartzmanGroup(\T^d,T_{A,b},\mu)$ gives the set of labels, one needs the additional assumption $\supp\mu = \T^d$. 

Let us explain how the labels in Theorem~\ref{t:affine:Tabschwartzgrp} arise. It is instructive to consider first the basic example $A = I$ so that $T$ is simply the identity map. In this case, one has $X(\Omega,T) = \T^{d+1}$. More generally, if $A$ fixes a subspace $V \subseteq \R^d$ of dimension $k$, this projects to a set $\bar{V}\subseteq \T^d$ that is homeomorphic to $\T^k$ and hence produces a copy of $\T^{k+1}$ inside $X$. This suggests a link between subspaces fixed by $A$ and generators of homotopy classes. The following lemma gives the precise result.

\begin{lemma} \label{lem:affine:Csharp}
Let $d \in \N$ and $A \in \SL(d,\Z)$ be given. For each $b \in \T^d$, define $X_b = X(\Omega,T_{A,b})$. For each $k \in K := 
\Z^d \cap \ker(I-A^*)$, $n \in \Z$, and $b \in \T^d$, define $g_{k,n,b} :X_b \to \T$ by
\begin{equation}
g_{k,n,b}[\omega,t] = \langle k, \omega \rangle + nt + t \langle k,b\rangle.
\end{equation}
\begin{enumerate}
\item[{\rm(a)}] For each $k \in K$,  $n \in \Z$, and $b \in \T^d$, $g_{k,n,b}$ is a well-defined continuous map $X_b\to\T$.\medskip

\item[{\rm(b)}] One has $C^\sharp(X_0,\T) = \{ [g_{k,n,0}] : k \in K, \ n \in \Z\}$; that is, every $g \in C(X_0,\T)$ is homotopic to some $g_{k,n,0}$.\medskip

\item[{\rm(c)}] For every $b \in \T^d$, $X_b$ is homeomorphic to $X_0$.
\medskip

\item[{\rm(d)}] For every $b \in \T^d$, one has $C^\sharp(X_b,\T) = \{ [g_{k,n,b}] : k \in K, \ n \in \Z\}$.
\end{enumerate} 
\end{lemma}

\begin{proof}(a) Since $k \in \ker(I-A^*)$, one has $A^*k =k$. Using this, one checks directly that 
\begin{align*}
g_{k,n,b}[T_{A,b}\omega,t] 
& = \langle k,A \omega + b \rangle + nt + t \langle k, b \rangle  \\
& = \langle A^* k,\omega \rangle + nt + (t+1) \langle k,b \rangle \\
& = \langle  k,\omega \rangle + n(t+1) + (t+1) \langle k,b \rangle  \\
& = g_{k,n,b}[\omega,t+1],
\end{align*}
which implies that $g_{k,n}$ is well-defined and continuous.\medskip

(b) Let $g \in C(X,\T)$ be given.  For $m \in \Z^d$, let $\chi_m(\omega) = \langle m,\omega \rangle$. 

\begin{step} \label{step:afftorus:kchoice} Write $\Omega_t$ for the fiber $\{[\omega,t] : \omega \in \Omega\}$ in $X$. Since every continuous map $\T^d\to \T$ is homotopic to $\chi_k:\omega \mapsto \langle k , \omega\rangle $ for some $k \in \Z^d$, there exists $k \in \Z^d$ such that $g|_{\Omega_0}$ is homotopic to $\chi_k$. By continuity, notice furthermore that $g|_{\Omega_t}$ is also homotopic to $\chi_k$ for every $t$. 
\end{step}

\begin{step}
Notice that $\langle m,A\omega\rangle = \langle A^*m,\omega\rangle$ for $m \in \Z^d$ and $\omega \in \T^d$. In particular, since $g|_{\Omega_0}$ is homotopic to $g|_{\Omega_1}$, we deduce $A^*k = k$ with $k$ from the previous step, and consequently, $k \in K$.
\end{step}

\begin{step} \label{step:afftorus:nchoice}
Considering the circle\footnote{This uses $A0=0$.} $\{[0,t] : t \in \R\}$, there exists $n \in \Z$ such that $t\mapsto g[0,t]$ is homotopic to $t\mapsto nt$.
\end{step}

\begin{step}
With $k$ and $n$ as in Steps~\ref{step:afftorus:kchoice} and \ref{step:afftorus:nchoice}, there is a map $h \in C(X,\T)$ homotopic to $g$ with $h[x,0] = \langle k,x \rangle$ and $h[0,t]=nt$. As mentioned before, note that $h|_{\Omega_t}$ is homotopic to $\langle k, \cdot \rangle$ for every $t$.
\end{step}

\begin{step}
Consider $h_0 = h - g_{k,n}$ with the $k$ and $n$ from Steps~\ref{step:afftorus:kchoice} and \ref{step:afftorus:nchoice}. By previous steps, $h_0$ vanishes on $\{[\omega,0] : \omega \in \Omega\} \cup\{[0,t] : t \in \R\}$. The reader can check that $h_0$ is then nullhomotopic.
\end{step}
Since $h_0$ is nullhomotopic, $[g]=[g_{k,n}]$ and the proof is done.
\medskip

(c) View $X_b$ as $\T^d \times [0,1]/\!\!\sim$ where $(\omega,1) \sim (T_{A,b} \,\omega,0)$. The map
\begin{equation}
\phi_b[\omega,t] = [\omega+tA^{-1}b,t]
\end{equation}
establishes the desired homeomorphism $\phi_b:X_b \to X_0$. Indeed, since
\begin{align*}
\phi_b[\omega,1]  = [\omega+A^{-1}b,1]
= [A(\omega+A^{-1}b),0] 
= [A\omega+b,0]
= \phi_b[A\omega + b,0]
= \phi_b[T_{A,b} \, \omega ,0],
\end{align*}
we see that $\phi_b$ is well-defined and continuous. The reader can check that the inverse is given by sending $[\omega,t] \in X_0$ to $[\omega-tA^{-1}b,t] \in X_b$ and is continuous.\medskip

(d) This follows from (b) by composing $g_{k,n,0}$ with the homeomorphism from (c).
\end{proof}

\begin{remark}
With some additional tools from algebraic topology, one can view Lemma~\ref{lem:affine:Csharp} as an explicit realization of the following diagram chase: First, since $\T$ is a $K(\Z,1)$-space, $C^\sharp(X,\T)$ can be identified with $H^1(X,\Z)$, the first cohomology group (cf.\ \cite{Hatcher2001:AlgTop}). Since $X$ is an orientable $d+1$-dimensional manifold, this can further be identified with $H_d(X,\Z)$, the $d$-th homology group via Poincar\'e duality. By using a standard exact sequence for mapping tori (cf.\ \cite[Example~2.48]{Hatcher2001:AlgTop}), one arrives at the exact sequence
\begin{equation}
0 \longrightarrow H_d(\T^d) \longrightarrow H_d(X) \longrightarrow \ker(I-A) \longrightarrow 0,
\end{equation}
from which one can realize $H_d(X)$ is free abelian on $1+\dim(\ker(I-A))$ generators. The functions $g_{k,n,b}$ provide explicit representatives of homotopy classes in $C^\sharp(X,\T)$ and hence (by chasing the diagram backwards) also give representatives of the cohomology classes in $H^1(X)$.
\end{remark}

\begin{lemma} \label{lem:affine:Amugknb}
Let $d \in \N$, $A \in \SL(d,\Z)$, and $b \in \T^d$ be given. For all $n \in \Z$, $k \in \Z^d \cap \ker(I- A^*)$, and $x \in X(\T^d,T_{A,b})$, one has $\rot(g_{k,n,b};x) = n+\langle k, b \rangle$.

In particular, if $\mu$ is any $T_{A,b}$-ergodic measure on $\T^d$, then
\begin{equation}
\mathfrak{A}_{\overline{\mu}}([g_{k,n,b}]) = n+ \langle k, b \rangle.
\end{equation}
for every $n \in \Z$ and $k \in \Z^d \cap \ker(I-A^*)$.
\end{lemma}

\begin{proof}
For the first claim, it suffices to establish $\rot(g_{k,n,b};[\omega,0]) = n + \langle k,b \rangle$ for every $\omega \in \T^d$.
Given $\omega \in \T^d$, the map $\phi_\omega:\R \to \T$ sending $t \in \R$ to $g_{k,n,b}[\omega,t]$ lifts to the map 
\[ \widetilde{\phi}_\omega: t \mapsto \langle k, \omega\rangle + nt + t\langle k,b \rangle \in \R.\]
One then sees immediately that
\[ \lim_{t \to \infty} \frac{\widetilde\phi_\omega(t)}{t} = n+\langle k, b \rangle. \]
Since this holds for every $\omega \in \T^d$, the first statement of the result follows. The second statement is an immediate consequence.
\end{proof}

We can now prove the main result.

\begin{proof}[Proof of Theorem~\ref{t:affine:Tabschwartzgrp}]
This is a consequence of Lemmas~\ref{lem:affine:Csharp} and \ref{lem:affine:Amugknb}.
\end{proof}

\begin{theorem}
Let $\Omega = \T^2$, $\alpha \in \T \setminus \Q$ be given, let $T:\T^2 \to \T^2$ denote the associated skew-shift given by
\begin{equation}
T(\omega_1,\omega_2) = (\omega_1 + \alpha,\omega_1 + \omega_2),
\end{equation}
and let $\mu$ be Lebesgue measure. One has
\begin{equation}
\schwartzmanGroup(\Omega,T) = \Z + \alpha\Z = \{n+m\alpha: n,m\in\Z\}.
\end{equation}
\end{theorem}

\begin{proof}
Since the skew-shift is strictly ergodic with Lebesgue measure supplying the unique invariant measure \cite{Furst1981Porter}, the result follows from Theorem~\ref{t:affine:Tabschwartzgrp}.
\end{proof}

Let us point out that the skew-shift is known to be uniquely ergodic (with Lebesgue measure the unique preserve measure), justifying the suppression of invariant measure from the notation.

\begin{theorem} \label{t:catmap}
Let $\Omega = \T^d$, suppose $A \in \SL(d,\Z)$ is such that no eigenvalue of $A$ is a root of unity, and let $\mu$ denote Lebesgue measure. One has
\begin{equation}
\schwartzmanGroup(\Omega,T_{A,0},\mu) = \Z.
\end{equation}
Consequently, for every $f \in C(\Omega,\R)$, $\Sigma_{\mu,f}$ is an interval.
\end{theorem}
\begin{proof}
The assumption on the eigenvalues of $A$ implies that $(\Omega,T_{A,0},\mu)$ is ergodic \cite[Corollary~1.10.1]{Walters1982:ErgTh}. Thus, the first statement follows from Theorem~\ref{t:affine:Tabschwartzgrp}. The second statement follows from Corollary~\ref{coro:A=Z}.
\end{proof}

In particular, the reader can readily check that the eigenvalues of 
\[ A = \begin{bmatrix} 2 & 1 \\ 1 & 1\end{bmatrix} \]
are $\frac{1}{2}(3\pm\sqrt{5})$, neither of which is a root of unity, and hence Theorem~\ref{t:catmap} applies to ergodic operators generated by the cat map.

\begin{remark}
There is a subtle point here: the ergodic measure must have full support in order for the gap labelling theorem to be applicable. Concretely, the reader may note that any hyperbolic toral automorphism has many invariant measures. Notably, any periodic point of such an automorphism may be used to generate an ergodic measure with finite support, which then leads to periodic potentials, which can in turn be chosen in such a way as to produce a spectral gap.\footnote{Indeed, the spectrum of any Schr\"odinger operator in $\ell^2(\Z)$ with a non-constant periodic potential has at least one gap \cite{Hochstadt1984LAA}.} 

On one hand, these measures generated by periodic orbits do not have full support, so the theorem asserting gaplessness does not apply to them. On the other hand, the mechanism at work here is the following: for the gap labelling theorem, one only works on the set of energies at which the associated cocycle is uniformly hyperbolic, which in particular means that the monodromy matrices associated with every periodic orbit must simultaneously be hyperbolic. On a single periodic orbit, we can have hyperbolicity of the associated monodromy without forcing uniform hyperbolicity of the whole cocycle.

\end{remark}

\bibliographystyle{abbrv}

\bibliography{gapbib}

\end{document}